\documentclass[10pt]{article}
\usepackage{amsfonts}
\usepackage{}
\usepackage{bm}
\usepackage{amsthm}
\usepackage{multirow}
\usepackage{authblk}

\renewcommand{\div}{\mathop{\rm div}\nolimits}

\newcommand\numberthis{\addtocounter{equation}{1}\tag{\theequation}}
\usepackage{amsmath}
\usepackage{amssymb}
\usepackage{graphicx}
\usepackage{color}

\usepackage{marginnote}
\usepackage{amstext,amsthm,amssymb,amsmath}
\usepackage{graphicx}
\usepackage{subfigure}

\newcommand{\om}{\Omega}

\newtheorem{theorem}{Theorem}[section]
\newtheorem{lemma}[theorem]{Lemma}

\newcommand{\eps}{\epsilon}

\newcommand{\norm}[1]{\left\|#1\right\|}

\newcommand{\strain}{\mbox{$\boldsymbol{\varepsilon}$}}

\newcommand{\avrg}[2]{{\left \langle #1 \right \rangle}_{#2}}

\hfuzz=\maxdimen
\graphicspath{ {./pic/}}
\title{Online Adaptive Local Multiscale Model Reduction for Heterogeneous Problems in Perforated Domains}

\author[1,2]{
Eric T. Chung \thanks{Email: {\tt tschung@math.cuhk.edu.hk}.} }
\author[2]{
Yalchin Efendiev \thanks{
 Email: {\tt efendiev@math.tamu.edu}.}}
 
\author[4]{
Wing Tat Leung}

\author[2,5]{
 Maria Vasilyeva}
  
  \author[4]{
 Yating Wang
}

\affil[1]{Department of Mathematics,
The Chinese University of Hong Kong (CUHK), Hong Kong SAR. }

\affil[2]{Department of Mathematics \& Institute for Scientific Computation (ISC),
Texas A\&M University,
College Station, TX 77843-3368, USA.}

\affil[3]{Department of Mathematics, Texas A\&M University, College Station, TX 77843-3368, USA.}

\affil[4]{Department of Computational Technologies, Institute of Mathematics and Informatics, North-Eastern Federal University, Yakutsk, 677980, Republic of Sakha (Yakutia), Russia. }

\begin{document}
\maketitle
%==============================

\begin{abstract}
In this paper, we develop and analyze an adaptive multiscale approach
 for heterogeneous problems in perforated domains.
We consider commonly used model problems including the Laplace equation,
the elasticity equation,
and the Stokes system in perforated regions.
In many applications, these problems have a multiscale
nature arising because of the perforations, their geometries, the sizes of the 
perforations, and
configurations.
 Typical modeling approaches extract average properties
in each coarse region, that encapsulate many perforations,
 and formulate a coarse-grid problem.
In some applications,
the coarse-grid problem can have a different form from the fine-scale
problem, e.g., the coarse-grid system
corresponding to a Stokes system in perforated domains leads to
Darcy equations on a coarse grid. In this paper, we present a general
offline/online procedure, which can adequately and adaptively
represent the local
degrees of freedom and derive appropriate coarse-grid equations.
Our approaches start with the offline procedure (following \cite{CELV2015}),
which constructs multiscale basis functions in each coarse region
and formulates coarse-grid equations. 
In \cite{CELV2015}, 
we presented the offline simulations without the analysis and adaptive
procedures, which are needed for accurate and efficient simulations.
The main contributions of
this paper are (1) the rigorous analysis of the offline approach (2) the
development of the online procedures and their analysis (3) the development
of adaptive strategies.
We present an online procedure,
which allows adaptively incorporating global information and is important for a fast convergence when
combined with the adaptivity.
We present online adaptive enrichment algorithms
for the three model problems mentioned above.
Our methodology allows adding and guides constructing new online multiscale basis
functions adaptively in appropriate regions.
We present the convergence analysis of the online adaptive enrichment algorithm
for the Stokes system.
In particular, we show that the
online procedure has a rapid convergence with a rate
related to the number of offline basis functions, and one
can obtain fast convergence by a sufficient
 number of offline basis functions,
which are computed in the offline stage.
The convergence theory can also be applied to the Laplace equation and the elasticity equation.
To illustrate the performance of our method, we present numerical results 
with both small and large perforations.
We see that only a few (1 or 2) online iterations can significantly
improve the offline solution.

\end{abstract}

\section{Introduction}

%Multiscale problems in perforated media and their importance.
%The need for upscaling and model reduction.

One important class of multiscale problems consists of problems
in perforated
domains (see Figure \ref{fig:perf_domain} for an illustration). 
In these problems, differential equations are formulated in perforated domains. These domains
can be considered the outside of inclusions or connected
bodies of various sizes.
Due to the variable sizes and geometries of these perforations, solutions to these problems have multiscale features.
One solution approach
involves posing the problem in a domain without perforations
but with a very high contrast penalty term representing the domain
heterogeneities (\cite{iliev2015modeling,verleye2008permeability, griebel2010homogenization, iliev2013numerical}).
However, the void space can be a small portion of the
whole domain and, thus, it is computationally expensive to
enlarge the domain substantially. 

Problems in perforated domains (\cite{sanchez1980non}),
as other multiscale problems,
require some model reduction techniques to reduce the computational cost. 
The main computational cost is due to the fine grid, which needs to resolve the space between the perforations.
There have been many homogenization
results  in perforated domains and for biphasic problems, where perforations can have
distinctly different properties, e.g.,
\cite{allaire1991homogenization,maz2010asymptotic,Jikov91,oleinik1996homogenization,fratrovic2015nonlinear,pankratova2015spectral,allaire2015upscaling,bare2015non,gilbert2014acoustic,muntean2012analysis,gilbert2006prototype,gilbert2003acoustic}. 
Homogenization approaches average microscale processes in perforations and outside and 
provide macroscale equations that differ from microscale equations. 
In the homogenization procedure, the local cell problems account for the
microscale interaction and are 
solved on a fine grid. Using the solutions of the local problems,
the effective properties can be computed.
The resulting homogenized equations can be solved on the coarse grid 
with the mesh size independent of
the size of the perforations for different boundary conditions and right hand sides.

To carry out the homogenization, typical assumptions on periodicity or
scale separation are needed to formulate the cell problems. Some generalization to
problems with random {\it homogeneous} pore-space geometries is introduced in a pioneering work
 \cite{beliaev1996darcy}, where the authors formulate assumptions, when homogenization 
can be done using representative volume element concepts. In these approaches, the cell
problems in very large domains are formulated and the effective properties are computed
using the solutions of the local problems. However, these approaches still assume that the
solution space can be approximated by the solutions of directional cell problems 
(i.e., $2$ cell problems in 2D) and the effective equations contain a limited number
of effective parameters (e.g., symmetric permeability tensor). 
These assumptions do not hold for general heterogeneities and
the effective properties may be richer (one may need more parameters). 
To study this, we use Generalized Multiscale Finite 
Element Method to identify necessary local cell solutions and obtain numerical macroscopic equations.

The main difference in developing
multiscale methods for problems in perforated domains is the complexity
of the domains and that many portions of the domain are excluded
in the computational domain. This poses a challenging task. For typical
upscaling and numerical homogenization (e.g., \cite{sanchez1980non, henning2009heterogeneous}),
the macroscopic equations
do not contain perforations and one computes the effective properties.
In multiscale methods, the macroscopic equations are numerically derived 
by computing multiscale basis functions \cite{Bris14, brown2014multiscale, CELV2015}.
Several multiscale methods have been developed for
problems in perforated domains. Our approaches are motivated
by recent works \cite{Maz'ya13, Bris14, Cao06, henning2009heterogeneous,Cao10,CELV2015}.
In this regard, we would like to mention recent works by Le Bris and his collaborators \cite{Bris14}, where
accurate multiscale basis functions are constructed.
These approaches differ from numerical homogenization and approaches
that use Representative Volume Element (RVE) \cite{ee03}.
However, these approaches do not contain a systematic way of enriching local multiscale spaces to 
obtain accurate macroscale representations of the underlying fine-scale problem.

%Discuss GMsFEM and its application in perforated domain.
%Mention that online is missing and why we need online.

Our proposed approaches are based on the Generalized Multiscale Finite
Element (GMsFEM) Framewowk\cite{GMsFEM13, AdaptiveGMsFEM, MixedGMsFEM}.
The GMsFEM follows the main concept of MsFEM
\cite{eh09,Iliev_MMS_11, Chu_Hou_MathComp_10, ab05, Babuska};
however, it systematically constructs multiscale basis
functions for each coarse block. The main idea of the GMsFEM is
to use local snapshot vectors (borrowed from global model reduction)
 to represent the solution space
and then identify local multiscale spaces by performing
appropriate local spectral problem. Using snapshot spaces
is essential in problems with perforations, because the snapshots contain
necessary geometry information. In the snapshot space, we perform
local spectral decomposition to identify multiscale basis functions.
These basis functions are derived based on the analysis presented in this paper.
The local multiscale basis functions obtained as a result represent the necessary degrees of
freedom to represent the microscale effects. This is in contrast to homogenization, where 
one apriori selects the number of cell problems.

We present the analysis of the proposed method.
We focus on analyzing Stokes
equations, since similar techniques can be easily extended to the elliptic and
the elasticity equations.
We note that in \cite{CELV2015}, we present the offline simulations for heterogeneous problems
in perforated domains.
In \cite{mixed-perforated}, the results for the mixed GMsFEM for the Laplace equation with
Neumann boundary conditions are presented.
The main contributions of
this paper are (1) the rigorous analysis of the offline approach (2) the
development of the online procedures and their analysis (3) the development
of adaptive strategies. 
We would like to emphasize that the adaptivity and online basis construction are important for
the success of multiscale methods. Indeed, in many regions, one may need only a few basis functions,
while some regions may require more degrees of freedom for approximating the solution space.
The online basis functions allow a fast convergence and takes into account global effects.

%Discuss your work. How do you construct online basis functions.
In the GMsFEM, the multiscale basis function construction is local
and uses both local snapshot solutions and local spectral problems.
In the paper, we discuss the use of randomized snapshots to reduce the offline cost associated
with the snapshot space computations.
One can use local oversampling techniques \cite{Efendiev_oversampling13}; however, the global
effects are still
not used. One can accelerate the convergence by computing multiscale
basis functions using a residual at the online stage \cite{chung2015residual,chan2015adaptive, ohlberger2015error}. This
is done by designing new multiscale basis functions, which solve
local problems using the global residual information. Online basis
functions are computed adaptively and only added in regions with largest
residuals. In this paper, we design online basis functions.
It is important that adding online basis function decreases
the error substantially and one can reduce the error in one iteration.
For this reason, constructing online basis functions
must guarantee that the error reduction is independent of small
scales and contrast.

%Convergence analysis.
Constructing online basis functions follows a rigorous analysis.
We show that if a sufficient number of offline multiscale basis functions
are chosen, one can substantially reduce the error. This reduction
is related to the eigenvalue that the corresponding eigenvector
is not included in the coarse space. Thus, one can get an estimate
of the error reduction apriori, which is important in practical simulations.
Our analysis for the offline procedure
 starts with the proof of the inf-sup condition, which shows the well-posedness of our scheme.
 Then, we derive an a-posteriori error bound for our GMsFEM.
 This bound shows that the error of the solution is bounded by a computable residual
 and an irreducible error. This irreducible error is a measure of approximating the fine-scale space
 by the snapshot space.
% We will next show that the online procedure is convergent up to the same irreducible error.
 We show that the convergence rate depends on the number of offline basis functions.
We note that in \cite{CELV2015}, we only present the offline simulation results without analysis. Based
on the analysis, we have modified some of multiscale basis functions for Stokes' equations
and moreover, introduced adaptive strategies and online basis construction techniques.

%Numerical result discussions.
In our numerical examples, we consider two different geometries,
where one case includes only a few perforations
 and the other case includes many perforations.
We considered elliptic, elasticity, and Stokes equations and only report
the results for elasticity and Stokes equations. Our results for
the offline consist of adding multiscale basis functions where we
observe that the error decreases as we increase the number
of basis functions.
However, the errors (especially those involving solution gradients)
can still be large.
For this reason, online basis functions are added, which can rapidly reduce the
error.
We summarize some of our quantitative
 results  below.
\begin{itemize}

\item
 For elasticity equations without adaptivity, we observe that, with
using $4$ offline basis functions per coarse neighborhood,
%(of resolution$648$),
we can achieve $7.4$ \% error in $L^2$ norm, while the error
is $26$ \% in $H^1$ norm. The results for the offline computations
are similar for two different geometries.

\item
For Stokes equations without adaptivity, we observe that, with
using $3$ offline basis functions per coarse block,
% (of resolution$936$)
we can achieve $0.94$ \% error in $L^2$ norm, while the error
is $8.8$ \% in $H^1$ norm. All errors are for the velocity field.
 The results for the offline computations
are better for the case with many inclusions.

\item
 For online simulations, we observe that the error decreases rapidly as
we add one online basis functions. The error keeps decreasing fast as
we increase the number of online basis functions; however, we are mostly
interested in error decay when one basis function is added. We observe
that the error decrease much faster if we have more than $1$ initial
offline basis function. For example, the error decreases only
$4$ times if one basis function is chosen, while the error
decreases more than $10$ times if $4$ initial basis functions
are selected (see Table \ref{tab:st-sh} and \ref{tab:st-ex} for the Stokes case
and second geometry).

\item
We observe that one can effectively use adaptivity to
reduce the computational cost in the online simulations.
Our adaptive results show that we can achieve better accuracy
for the same number of online basis functions.

\end{itemize}

The paper is organized as follows. In Section \ref{prelim},
we present a general setting for perforated problems,
the coarse and fine grid definitions, and a general idea of the GMsFEM.
In Section \ref{Construction},
we discuss constructing offline and online basis
functions. Section \ref{numerical} is devoted to numerical results. In Section \ref{Analysis},
we present the convergence analysis for the offline and
online GMsFEM. The conclusions are presented in Section \ref{conclusion}.

%structure
\section{Preliminaries}
\label{prelim}
%==============================

\subsection{Problem setting}

In this section, we present the underlying problem as stated in \cite{CELV2015,mixed-perforated}
and the corresponding fine-scale and coarse-scale discretization.
Let $\Omega\subset \mathbb{R}^d$ ($d=2,3$) be a bounded domain covered
by inactive cells (for Stokes flow and Darcy flow) or active cells (for elasticity problem) $\mathcal {B}^{\eps}$.  
In the paper, we will consider $d=2$ case, though our results can be extended
to $d>2$.
We use the superscript $\eps$ to denote quantities related to perforated domains.
The active cells are where the underlying problem is solved, while inactive cells are the rest of the region.
Suppose the distance between inactive cells (or active cells) is
of order $\eps$. Define $\Omega^{\eps}:=\Omega\backslash \mathcal {B}^{\eps}$, assume it is polygonally bounded. See Figure \ref{fig:perf_domain} for an illustration of the perforated domain.
We consider the following problem defined in a perforated domain $\Omega^{\eps}$
\begin{align} \label{eq:original}
&\mathcal {L}^{\eps}(w)=f, \quad \text{in} \quad \Omega^{\eps},&\\
&w=0 \text{ or } \frac{\partial w}{\partial n}=0, \text{ on }\partial \Omega^{\eps}\cap  \partial \mathcal {B}^{\eps}, \label{eq:o2} &\\
&w=g, \text{ on }\partial \Omega\cap \partial \Omega^{\eps}, \label{eq:o3} &
\end{align}
where $\mathcal {L}^{\eps}$ denotes a linear differential operator, $n$ is the unit outward normal to the boundary, $f$ and $g$ denote given functions with sufficient regularity.

\begin{figure}[htb]
  \centering
  \includegraphics[width=0.3 \textwidth]{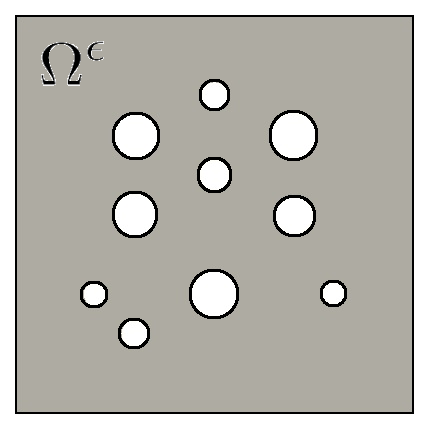}
  \caption{Illustration of a perforated domain.}
  \label{fig:perf_domain}
\end{figure}

Denote by ${V}(\om^{\eps})$ the appropriate solution space, and
  \[{V}_{0}(\om^{\eps})=\{v\in {V}(\om^{\eps}), v=0 \text{ on }\partial\om^{\eps}\}.\]
The variational formulation of Problem \eqref{eq:original}-\eqref{eq:o3} is to find $w\in {V}(\om^{\eps})$ such that
\[
\avrg{\mathcal {L}^{\eps}(w),v}{\om^{\eps}}= (f,v)_{\om^{\eps}} \qquad  \text{for all } v \in V_{0}(\om^{\eps}),
\]
where $\avrg{\cdot,\cdot}{\om^{\eps}}$ denotes a specific for the application inner product over $\om^{\eps}$ for either scalar functions or vector functions, and 
and $(f,v)_{\om^{\eps}}$ is the $L^2$ inner product.
Some specific examples for the above abstract notations are given below.

\textbf{Laplace:} For the Laplace operator with homogeneous Dirichlet boundary conditions on $\partial \om^{\eps}$, we have
\begin{align}\label{eqn:laplace}
\mathcal {L}^{\eps}(u)=-\Delta u,
\end{align}
and ${V}(\om^{\eps})=H^{1}_{0}(\Omega^{\eps})$, $\avrg{\mathcal {L}^{\eps}(u),v}{\om^{\eps}}=(\nabla u,\nabla v)_{\om^{\eps}}$.

\textbf{Elasticity:} For the elasticity operator with a
homogeneous Dirichlet boundary condition on $\partial \om^{\eps}$, we assume the medium is isotropic.
Let $ {u}\in (H^{1}(\Omega^{\eps}))^{2}$ be the displacement field.
The strain tensor $ {\strain}( {u})\in (L^{2}(\Omega^{\eps}))^{2\times 2}$
is defined by
\begin{equation*}
 {\strain }( {u}) = \frac{1}{2} ( \nabla  {u} + \nabla  {u}^T ).
\end{equation*}
Thus, the stress tensor $ {\sigma}( {u})\in (L^{2}(\Omega^{\eps}))^{2\times 2}$ relates to the strain tensor $ {\strain}( {u})$ such that
\begin{equation*}
 {\sigma}(u)= 2\mu  {\strain} + \xi \nabla\cdot  {u} \,  {I},
\end{equation*}
where $\xi >0$ and $\mu>0$ are the Lam\'e coefficients. We have
\begin{align}\label{eqn:elasticity}
\mathcal {L}^{\eps}(u)=- \nabla \cdot  {\sigma},
\end{align}
where ${V}(\om^{\eps})=(H^{1}_{0}(\Omega^{\eps}))^{2}$ and
$\avrg{\mathcal {L}^{\eps}(u),v}{\om^{\eps}}
=2\mu(\strain(u),\strain(v))_{\om^{\eps}}+\xi (\nabla\cdot u,\nabla\cdot v)_{\om^{\eps}}$.

\textbf{Stokes:} For Stokes equations, we have
\begin{align}\label{eqn:Stokes}
\mathcal {L}^{\eps}(u\;,p)=\begin{pmatrix}
\nabla p -\Delta {u}\\
\nabla \cdot {u}
\end{pmatrix},
\end{align}
where $\mu$ is the viscosity, $p$ is the fluid pressure, $u$ represents the velocity, ${V}(\om^{\eps})=(H^{1}_{0}(\Omega^{\eps}))^{2}\times L^{2}_{0}(\Omega^{\eps})$, and
\[\avrg{\mathcal {L}^{\eps}(u\;,p),(v\;,q)}{\om^{\eps}}=
\begin{pmatrix}
(\nabla u,\nabla v)_{\om^{\eps}} &-(\nabla \cdot v,p)_{\om^{\eps}}\\
(\nabla \cdot u,q)_{\om^{\eps}}&0
\end{pmatrix}.\]
We recall that $L^{2}_{0}(\Omega^{\eps})$ contains functions in $L^{2}(\Omega^{\eps})$
with zero average in $\Omega^{\eps}$.

In this paper, we will show the results for elasticity and Stokes equations.
The results for Laplace have similar convergence analysis and computational results as those for elasticity equations, so we will omit them here.

\subsection{Coarse and fine grid notations}

For the numerical approximation of the above problems, we first introduce the notations of fine and coarse grids.
Let $\mathcal{T}^H$ be a coarse-grid partition of the domain
$\Omega^{\eps}$ with mesh size $H$. Here, we assume that the perforations will not split the coarse triangular element,
as in this case, the coarse block will have two disconnected regions. In general, the proposed
concept can be applied to this disconnected case; however, for simplicity, we avoid it and assume that 
every coarse-grid block is path-connected (i.e., any two points can be connected within the coarse block).
Notice that, the edges of the coarse elements do not necessarily have straight edges
because of the perforations (see Figure~\ref{fig：nbhd_def}).
By conducting a conforming
refinement of the coarse mesh $\mathcal{T}^H$, we can obtain
 a fine mesh $\mathcal{T}^h$ of $\Omega^{\eps}$ with mesh size $h$.
 %where we assume that the fine mesh resolves the multiscale perforations.
Typically, we assume that $0 < h \ll H < 1$, and that the fine-scale mesh $\mathcal{T}^h$
is sufficiently fine to fully resolve the small-scale information of the domain, and $\mathcal{T}^H$ is a coarse mesh containing many fine-scale features.
Let $N_v$ and $N_e$ be the number of nodes and edges in coarse grid respectively.
We denote by $\{x_i|1\leq i \leq N_v\}$ the set of coarse nodes, and $\{E_j| 1\leq j\leq N_e \}$ the set of coarse edges.

\begin{figure}[!h]
  \centering
  \includegraphics[width=0.6 \textwidth]{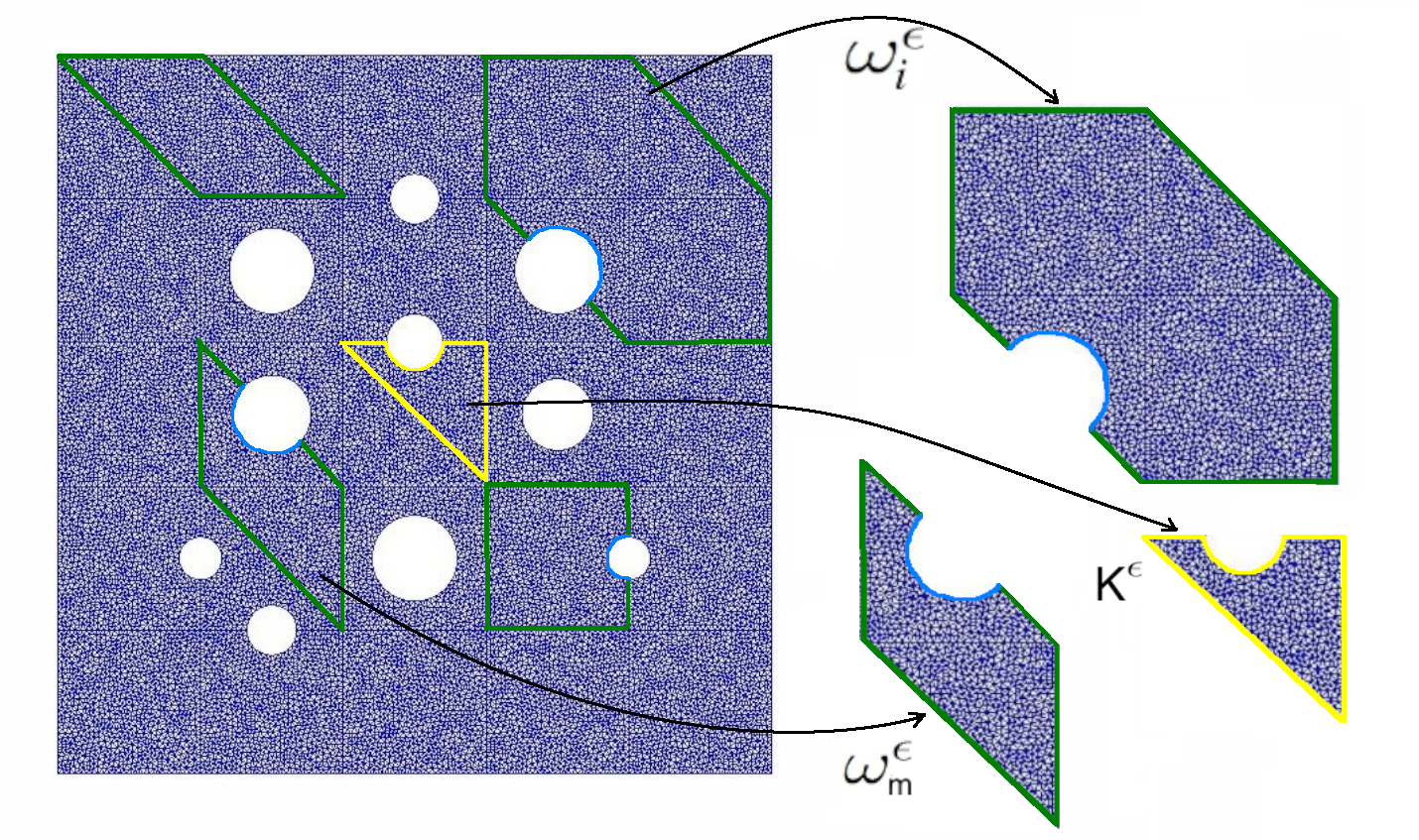}
  \caption{Illustration of coarse elements and coarse neighborhoods.}
  \label{fig：nbhd_def}
\end{figure}

For all the three model problems, we define a coarse neighborhood $\omega_i^{\eps}$ for each coarse node $x_i$ by
\begin{equation}
\label{eq:n1}
\omega_i^{\eps} = \cup\{K_j^{\eps}\in \mathcal{T}^H; ~~x_i\in \bar{K_j^{\eps}} \},
\end{equation}
which is the union of all coarse elements having the node $x_i$.
For the Stokes problem, additionally, we define a coarse neighborhood $\omega_m^{\eps}$ for each coarse edge $E_m$ by
\begin{equation}
\label{eq:n2}
\omega_m^{\eps} = \cup\{K_j^{\eps}\in \mathcal{T}^H; ~~E_m\in \bar{K_j^{\eps}} \},
\end{equation}
which is the union of all coarse elements having the edge $E_m$.
See Figure \ref{fig：nbhd_def} for an illustration of the coarse neighborhoods.

On the triangulation $\mathcal{T}^h$,
we introduce the following finite element spaces
\begin{align*}
{V}_h  &:= \{ {v} \in V(\Omega^{\eps}):\quad {v}|_K \in (P^k(K))^l
\mbox{ for all } K \in \mathcal{T}^h \},
\end{align*}
where, $P^k$ denotes the polynomial of degree $k$( $k=0,\;1,\;2$), and $l$( $l=1,\;2$)
indicates either a scalar or a vector.
Note that for the Laplace and elasticity operators, we choose $k=1$, i.e., piecewise linear function space as our fine-scale approximation space; for Stokes problem, we use $(P^2(K))^2$ for fine-scale velocity approximation and $P^0(K)$ for fine-scale pressure approximation.
We use $Q_h$ to denote the space for pressure.

We will then obtain the fine-scale solution $u\in V_h$ by solving the following variational problem
\begin{align}\label{eq:fine_system}
\avrg{\mathcal {L}^{\eps}(u),v}{\om^{\eps}}= (f,v)_{\om^{\eps}}, \qquad  \text{for all } v \in V_h
\end{align}
for Laplace and elasticity, and
obtain the fine-scale solution $(u,p)\in V_h \times Q_h$ by solving the following variational problem
\begin{align}\label{eq:fine_system1}
\avrg{\mathcal {L}^{\eps}(u,p),(v,q)}{\om^{\eps}}= ((f,0),(v,q))_{\om^{\eps}}, \qquad  \text{for all } (v,q) \in V_h \times Q_h
\end{align}
for the Stokes system. These solutions are used as reference solutions to test the performance of our schemes.

\subsection{General idea of GMsFEM}\label{GMsFEM}

Now, we present the general idea of GMsFEM
\cite{GMsFEM13,hw97, chung2015residual}. We divide the computations into
offline and online stages.

\textbf{Offline stage.} The construction of offline space usually contains two steps:
\begin{itemize}
\item Construction of a snapshot space that will be used to compute an offline space.
\item Construction of a small dimensional offline space by performing a dimension reduction in the snapshot space.
\end{itemize}

From the above process, we will get a set of basis functions $\{\Psi_i^{\text{off}}\}$ such that
each $\Psi_i^{\text{off}}$ is supported in some coarse neighborhood $w_l^{\eps}$.
Also, the basis functions satisfy a partition of unity property.

Once the bases are constructed, we define the coarse function space as
$${V}_{\text{off}} := \mbox{span}\{{\Psi}_i^{\text{off}}\}_{i=1}^{M},$$
where $M$ is the number of coarse basis functions.

In the offline stage of GMsFEM, we seek an approximation ${u}_{\text{ms}}=\sum_{i=1}^{M} c_i{\Psi}_i^{\text{off}}$ in $V_{\text{off}}$,
 which satisfies the coarse-scale offline formulation,
\begin{align}\label{eq:coarse_system}
\avrg{\mathcal {L}^{\eps}(u_{\text{ms}}),v}{\om^{\eps}}= (f,v)_{\om^{\eps}}, \qquad  \text{for all } v \in V_{\text{off}}.
\end{align}
Here, the bilinear forms $\avrg{\mathcal {L}^{\eps}(u_{\text{ms}}),v}{\om^{\eps}}$ are as defined before,
and $(f,v)_{\om^{\eps}}$ is the $L^2$ inner product.

\textbf{Online stage.} Now, we will turn our attention to the online computation.
 At the enrichment level $m$,
denote by ${V}_{\text{ms}}^m$ and ${u}_{\text{ms}}^m$ the corresponding GMsFEM space and solution, respectively.
The online basis functions are constructed based on the residuals of the current multiscale solution ${u}_{\text{ms}}^m$.
To be specific, one can compute the local residual $R_i=(f,v)_{\omega_i^{\eps}}-\avrg{\mathcal {L}^{\eps}(u_{\text{ms}}^m),v}{\omega_i^{\eps}}$ in each coarse neighborhood $\omega_i^{\eps}$. For the coarse neighborhoods where the residuals are large,
we can add one or more basis functions by solving
\[
\mathcal {L}^{\eps}(\phi_i^{\text{on}}) = R_i.
\]

 Adding the online basis in the solution space, we will get a new coarse function space $V_{\text{ms}}^{m+1}$.
The new solution $u_{\text{ms}}^{m+1}$ will be found in this approximation space.
This iterative process is stopped when some error tolerance is achieved.
The accuracy of the GMsFEM relies on the coarse basis functions.
We shall present the construction of suitable basis functions in both offline and online stages for the differential operators defined above.

%------------------------------------
\section{The construction of offline and online basis functions}
\label{Construction}
%------------------------------------

In this section, we describe the construction of offline and online basis for elasticity problem and Stokes problem.

In the offline computation, we first construct a snapshot space $V_{\text{snap}}^{i}$ for each coarse neighborhood $\omega_i^{\eps}$.
Construction of the snapshot space involves solving the local problems for various choices of input parameters.
The offline space $V_{\text{off}}$ is then constructed via a dimension reduction in the snapshot space using an auxiliary
spectral decomposition. The main objective is to seek a subspace of the snapshot space such that it can approximate
any element of the snapshot space in an appropriate sense defined via auxiliary bilinear forms.
Based on the residual of the current solution, we enrich the solution space by adding some online functions to enhance the accuracy of the solution.
The precise construction of offline and online basis will be presented for different applications.

\subsection{Elasticity Problem}

In this section, we will consider the elasticity problem \eqref{eqn:elasticity} with a homogeneous Dirichlet boundary condition.

\subsubsection{Snapshot Space}
The snapshot space for elasticity problem consists of extensions of the fine-grid functions $\delta_k^h$ in $\omega_i^{\eps}$.
Here $\delta_k^h = 1$ at the fine node $x_k\in \partial{\omega_i^{\eps}}\backslash\partial \mathcal {B}^{\eps}$, $\delta_k^h = 0$ at other fine nodes $x_j\in \partial{\omega_i^{\eps}}\backslash \partial \mathcal {B}^{\eps}$,  and $\delta_k^h = 0$ in $\partial \mathcal {B}^{\eps}$.
Let $V_h^i$ be the restriction of the fine grid space $V_h$ in $\omega_i^{\eps}$
and $V_{h,0}^i \subset V_h^i$ be the set of functions that vanish on $\partial\omega_i^{\eps}$.
We will find $u_k^i \in V_h^i$ with $supp(u_k^i) \subset \omega_i^{\eps}$ by solving the following problems on a fine grid
\begin{equation}\label{hrmext-elas}
\int_{\omega_i^{\eps}} \Big( 2\mu \strain(u_k^i):\strain(v)+\xi \nabla\cdot u_k^i\nabla\cdot v \Big)dx = 0,  \quad \forall v \in V^i_{h,0},\\
\end{equation}
with boundary conditions
\[u_k^i = 0~~ \text{on}~~  \partial{\omega_i^{\eps}}\cap \partial\mathcal {B}^{\eps},\ \ \ \ u_k^i = (\delta_j^{i},0) ~~\text{or}~~(0,\delta_j^{i})~~ \text{on} ~~ \partial{\omega_i^{\eps}}.
\]
%Note that for the elasticity operator with Neumann boundary conditions $ \displaystyle\frac{\partial u}{\partial n}=0$, we will use %Neumann boundary conditions instead of the local problems proposed above.
We will collect the solutions of the above local problems to
generate the snapshot space. Let $\psi_{k}^{i,\text{snap}}:= u_k^i$ and define the snapshot space by
\[
V_{\text{snap}} = \text{span}\{\psi_{k}^{i,\text{snap}}:\ \ 1\leq k \leq J_i, \, 1\leq i \leq N_v\},
\]
where $J_i$ is the number of snapshot basis in $\omega_i^{\eps}$, and $N_v$ is the number of nodes.
To simplify notations, let $M_{\text{snap}} = \sum_{i=1}^{N}J_i$ and write
\[
V_{\text{snap}} = \text{span}\{\psi_{i}^{\text{snap}}:\ \ 1\leq i \leq M_{\text{snap}}\}.
\]

\subsubsection{Offline space}

This section is devoted to the construction of the offline space via a  spectral decomposition.
We will consider the following eigenvalue problems in the space of snapshots:
\begin{equation}
A^{i,\text{off}} \Psi_k^{i,\text{off}} = \lambda_k^{i,\text{off}} S^{i,\text{off}}\Psi_k^{i,\text{off}},  \label{offeig1}
\end{equation}
where
\begin{equation}
\begin{split}
 A^{i,\text{off}}&= a_i(\psi_m^{i,\text{snap}} , \psi_n^{i,\text{snap}} )= \int_{\omega_i^{\eps}}
\Big( 2\mu  {\strain}( \psi_m^{i,\text{snap}}) :  {\strain}( \psi_n^{i,\text{snap}})
+ \xi \nabla\cdot  {\psi_m^{i,\text{snap}}} \, \nabla\cdot  \psi_n^{i,\text{snap}}
\Big) \;, \\
\displaystyle S^{i,\text{off}}&= s_i(\psi_m^{i,\text{snap}} , \psi_n^{i,\text{snap}} ) = \int_{\omega_i^{\eps}}  (\xi + 2 \mu)  \psi_m^{i,\text{snap}} \cdot  \psi_n^{i,\text{snap}}.
\end{split}
\end{equation}
We assume that the eigenvalues are arranged in the increasing order.
To simplify notations, we write $\lambda_k^i = \lambda_k^{i,\text{off}}$.

To generate the offline space, we choose the smallest $M_i$ eigenvalues from Equation~\eqref{offeig1} and form the corresponding eigenfunctions in the respective snapshot spaces  by setting
$\Phi_k^{i,\text{off}} = \sum_j \Psi_{kj}^{i,\text{off}} \psi_j^{i,\text{snap}}$, for $k=1,\ldots, M_i$, where $\Psi_{kj}^{i,\text{off}}$ are the coordinates of the vector $\Psi_{k}^{i,\text{off}}$.
The offline space is defined as the span of $\chi_i \Phi^{i,\text{off}}_k$, namely,
%where $\chi_i$
%is a partition of unity function for the node $i$
\[
V_{\text{off}} = \text{span} \{\chi_i \Phi_{l}^{i,\text{off}}:\ \ 1\leq l \leq l_i, \, 1\leq i \leq N_v \},
\]
where $l_i$ is the number of snapshot basis in $\omega_i^{\eps}$,
and $\{ \chi_i \}$ is a set of partition of unity functions for the coarse grid.
One can take $\{ \chi_i \}$ as the standard hat functions or standard multiscale basis functions.
To simplify notations further, let $M = \sum_{i=1}^{N}l_i$ and write
\[
V_{\text{off}} = \text{span}\{\chi_i\Phi_{i}^{\text{off}}:\ \ 1\leq i \leq M \}.
\]

\subsubsection{Online adaptive method} \label{sec:online_algorithm}

By the offline computation,  we construct multiscale basis functions that can  be used for any input parameters to solve the problem on the coarse grid. In the earlier works \cite{chung2015online, chung2015residual}, the online method for the diffusion equation with heterogeneous coefficients has been proposed. In this section, we consider the construction of the online basis functions for elasticity problem in perforated domains and present an adaptive enrichment algorithm.
We use the index $m \geq 1$ to represent the enrichment level. The online basis functions are computed based on some local residuals for the current multiscale solution $u_{\text{ms}}^m\in V^m_{\text{ms}}$, where we use $V^m_{\text{ms}}$ to denote the corresponding space that can contain both offline and online basis functions.

Let $V^{m+1}_{\text{ms}} = V^{m}_{\text{ms}} + \text{span} \{ \phi^{\text{on}} \}$ be the new approximate space that constructed by adding online basis $\phi^{\text{on}} \in {V}^i_{h,0}$ on the $i$-th coarse neighborhood $\omega_i^{\eps}$.
For each coarse grid neighborhood $\omega_i^{\eps}$, we define the residual $R_i$ as a linear functional on $V^i_{h,0}$ such that
\[
R_i(v) =  \int_{\omega_i^{\eps}} f v dx -
\int_{\omega_i^{\eps}} \Big(2\mu\strain(u^m_{\text{ms}}):\strain(v)+ \xi \nabla\cdot u^m_{\text{ms}}\nabla\cdot v\Big) dx,  \quad \forall v \in V^i_{h,0}.
\]
The norm of $R_i$ is defined as
\[
||R_i||_{(V^i_h)^*} = \sup_{v \in V^i_{h,0} } \frac{|R_i(v)|}{a_i(v,v)^{\frac{1}{2}}},
\]
where $a_i(v,v) = \int_{\omega_i^{\eps}} \Big(2\mu\strain(v):\strain(v)+ \xi \nabla\cdot v\nabla\cdot v\Big) dx$.

For the computation of this norm, according to the Riesz representation theorem, we can first compute $\phi^{\text{on}}$ as the solution of following problem
\begin{equation}
\label{online}
\int_{\omega_i^{\eps}} \Big(2\mu\strain(\phi^{\text{on}}):\strain(v)+ \xi \nabla\cdot \phi^{\text{on}}\nabla\cdot v\Big) dx = \int_{\omega_i^{\eps}} f v \, dx -
\int_{\omega_i^{\eps}} \Big(2\mu\strain(u^m_{\text{ms}}):\strain(v)+ \xi \nabla\cdot u^m_{\text{ms}}\nabla\cdot v\Big) dx , \quad \forall v \in V^i_{h,0}
\end{equation}
and take $||R_i||_{(V^i_h)^*} = a_i(\phi^{\text{on}},\phi^{\text{on}})^{\frac{1}{2}}$.

For the construction of the adaptive online basis functions,
we use the following error indicators to access the quality of the solution.
In those non-overlapping coarse grid neighborhoods $\omega_i^{\eps}$ with large residuals, we enrich the space by
finding online basis $\phi^{\text{on}} \in {V}^i_{h,0}$ using equation \eqref{online}.
%we first choose $0<\theta<1$. Then for each non-overlapping coarse grid neighborhood $\omega_i^{\eps}$, we find the online basis $\phi^{\text{on}} \in {V}^i_h$ using equation \eqref{online}.
%
%In our numerical simulations, we will use following error indicators:
\begin{itemize}
\item \textit{Indicator 1.} The error indicator based on local residual
\begin{equation}  \label{itm:ind1}
\eta_i =||R_i||^2_{(V^i_h)^*}
\end{equation}
\item \textit{Indicator 2.} The error indicator based on local residual and eigenvalue
\begin{equation}  \label{itm:ind2}
\eta_i = \left( \lambda^{\omega_i}_{l_i+1} \right)^{-1} ||R_i||^2_{(V^i_h)^*}
\end{equation}
\end{itemize}

Now we present the adaptive online algorithm.
We start with enrichment iteration number $m=0$ and choose $\theta \in (0, 1)$. Suppose the initial number of offline basis functions is $l_i^m$($m=1$) for each coarse grid neighborhood $\omega_i^{\eps}$, and the multiscale space is $V^m_{\text{ms}}$($m=1$). For $m=1,2,...$
\begin{itemize}
\item \textit{Step 1.} Find $u^m_{\text{ms}}$ in  $V^{m}_{\text{ms}}$ such that
\[
\begin{split}
\int_{\omega_i^{\eps}} \Big(2\mu\strain(u^m_{\text{ms}}):\strain(v)+ \xi \nabla\cdot u^m_{\text{ms}}\nabla\cdot v\Big) dx = \int_{\omega_i^{\eps}} fv, \quad \forall v \in V^m_{\text{ms}}.  \\
\end{split}
\]

\item \textit{Step 2.} Compute error indicators ($\eta_i$) for every coarse grid neighborhoods $\omega_i^{\eps}$ and sort them in decreasing order
$\eta_1 \geq \eta_2 \geq  ... \geq \eta_{N}$.

\item \textit{Step 3.} Select coarse grid neighborhoods $\omega_i^{\eps}$, where enrichment is needed. We take smallest $k$ such that
\[
\theta \sum_{i=1}^{N_v} \eta_i \leq \sum_{i=1}^{k} \eta_i.
\]

\item \textit{Step 4.} Enrich the space by adding online basis functions. For each $\omega_i^{\eps}$, where $i = 1,2,...,k$, we find $\phi^{\text{on}} \in V^i_{h,0}$ by solving \eqref{online}. The resulting space is denoted by $V^{m+1}_{\text{ms}}$.
\end{itemize}

We repeat the above procedure until the global error indicator is small or we have certain number of basis functions.

\subsection{Stokes problem}

In the above section, we presented the online procedure for the
elasticity equations.
In this section, we present the constructions of snapshot, offline and online basis functions for the Stokes problem.

\subsubsection{Snapshot space}

Snapshot space is a space which contains an extensive set of basis functions that are solutions of local problems with all possible boundary conditions up to fine-grid resolution. To get snapshot functions, we solve the following problem on the coarse neighborhood $\omega_i^{\eps}$: find $(u^i_l, p^i_l)$ (on a fine grid) such that
\begin{equation}
\label{eq:hrmext}
\begin{split}
\int_{\omega_i^{\eps}}   \nabla u^i_l  : \nabla v  dx
- \int_{\omega_i^{\eps}} p^i_l \div(v) dx &= 0,  \quad \forall v \in V^i_{h,0},\\
\int_{\omega_i^{\eps}} q \div(u^i_l)dx &= \int_{\omega_i^{\eps}} c  q dx,  \quad \forall q \in Q^i_h,
\end{split}
\end{equation}
with boundary conditions
\[
u^i_l  =  (0, 0), \text{ on } \partial \mathcal {B}^{\eps}, \quad
u^i_l  = (\delta^i_l, 0)  \text{ or } (0, \delta^i_l) ,  \text{ on } \partial \omega_i^{\eps} \backslash  \partial \mathcal {B}^{\eps},
\]
where function $\delta^i_l$ is a piecewise constant function such that  it has value $1$ on $e_l$ and value $0$ on other fine-grid edges.  Notice that $\omega_i^{\eps} \backslash  \partial \mathcal {B}^{\eps} = \cup_{l=1}^{S_i} e_l$, where $e_l$ are the fine-grid edges and $S_i$ is the number of these fine grid edges on $\omega_i^{\eps} \backslash  \partial \mathcal {B}^{\eps}$.
In (\ref{eq:hrmext}), we define $V_h^i$ and $Q_h^i$ as the restrictions of the fine grid space in $\omega_i^{\eps}$
and $V_{h,0}^i \subset V_h^i$ be functions that vanish on $\partial\omega_i^{\eps}$. Notice that $u^i_l$ and $p^i_l$ are supported in $\omega_i^{\eps}$.
We remark that the constant $c$ in \eqref{eq:hrmext} is chosen by compatibility condition, $c = \frac{1}{|\omega_i^{\eps}|} \int_{\partial \omega_i^{\eps} \backslash  \partial \mathcal {B}^{\eps} } u^i_l \cdot n_i \, ds$.
%The above local problem is solved numerically on the fine-grid defined in $\omega_i^{\eps}$  by the lowest-order Crouzeix-Raviart element.
We emphasize that, for the Stokes problem, we will solve (\ref{eq:hrmext})
in both node-based coarse neighborhoods (\ref{eq:n1}) and edge-based coarse neighborhoods (\ref{eq:n2}).

The collection of the solutions of above local problems generates the snapshot space, $\psi_l^{i, \text{snap}} = u^i_l$ in $\omega_i^{\eps}$:
\[
V_{\text{snap}} = \{ \psi_l^{i, \text{snap}}: 1 \leq l \leq 2 S_i, \, 1 \leq i \leq (N_e + N_v) \},
\]
where we recall that $N_e$ is the number of coarse-grid edges and $N_v$ is the number of coarse-grid nodes.

\subsubsection{Offline Space}

We perform a space reduction in the snapshot space through the use of a local spectral problem in $\omega_i^{\eps}$.  The purpose of this is to determine the dominant modes in the snapshot space and to obtain a small dimension space for the approximation the solution.

We consider the following local eigenvalue problem in the snapshot space
\begin{equation}\label{eq:offeig2}
A^{i,\text{off}} \Psi_k =
\lambda_k^{i,\text{off}} S^{i,\text{off}} \Psi^{i,\text{off}}_k,
\end{equation}
where
\[
 A^{i,\text{off}}
= a_i(\psi_m^{i,\text{snap}} , \psi_n^{i,\text{snap}} )
\]\[
S^{i,\text{off}}
= s_i(\psi_m^{i,\text{snap}} , \psi_n^{i,\text{snap}} )
\]
and
\[
a_i(u, v) = \int_{\omega_i^{\eps}} \nabla u : \nabla v dx, \quad \text{ and } \quad
s_i(u, v) = \int_{\omega_i^{\eps}} |\nabla \chi_i |^2 u \cdot v \, dx
\]
and $\chi_i$ will be specified later.
Note that the above spectral problem is solved in the local snapshot space
corresponding to the neighborhood domain $\omega_i^{\eps}$.
We arrange the eigenvalues in the increasing order, and
choose the first $M_i$ eigenvalues and take the corresponding eigenvectors
$\Psi_k^{i,\text{off}}$, for $k = 1,2,...,M_i$,
to form the basis functions, i.e.,
$\widetilde{\Phi}_k^{i,\text{off}} = \sum_j \Psi_{kj}^{i,\text{off}} \psi_j^{i,\text{snap}}$, where $\Psi_{kj}^{i,\text{off}}$ are the coordinates of the vector $\Psi_{k}^{i,\text{off}}$.
We define
\[
\widetilde{V}^i_{\text{off}} = \text{span} \{ \widetilde{\Phi}_k^{i,\text{off}}, ~~k = 1,2,..., 2S_i \}. \numberthis\label{locofft}
\]

For construction of conforming offline space, we need to multiply the functions $\widetilde{\Phi}_k^{i,\text{off}}=(\widetilde{\Phi}_{x_1, k}^{i,\text{off}},\widetilde{\Phi}_{x_2, k}^{i,\text{off}})$ by a partition of unity function $\chi_i$.
We remark that the partition of unity functions $\{ \chi_i\}$ are defined with respect to the coarse nodes
and the mid-points of coarse edges. One can choose $\{ \chi_i\}$ as the standard multiscale finite element basis.
%\in (P^2(\omega_i^{\epsilon}))^d$.
However, upon multiplying by partition of unity functions, the resulting basis functions do not have constant divergence any more,
which affects the stability of the scheme.
%this may ruin the divergence constant property of functions $\Psi_k^{i,\text{off}}$.
To resolve this problem, we solve two local optimization problems in every coarse element $K^i_j \subset \omega_i^{\eps}$:
\begin{equation}\label{hrm-x}
\min{ \norm{\nabla \Phi_{x_1, k}^{i,\text{off}}}_{L^2(K^i_j)} } \text{~~such that~~}
\div( \Phi_{x_1, k}^{i,\text{off}})  =  \frac{1}{|K^i_j|} \int_{\partial K^i_j }  (\chi_i \widetilde{\Phi}_{x_1, k}^{i,\text{off}} , 0) \cdot n_i \, ds, \quad \text{in } K^i_j
\end{equation}
with $\Phi_{x_1, k}^{i,\text{off}}  =   (\chi_i \widetilde{\Phi}_{x_1, k}^{i,\text{off}} , 0), \text{ on } \partial K^i_j$,
and
\begin{equation}\label{hrm-y}
\min{ \norm{\nabla \Phi_{x_2, k}^{i,\text{off}}}_{L^2(K^i_j)} } \text{~~such that~~}
\div( \Phi_{x_2, k}^{i,\text{off}})  =  \frac{1}{|K^i_j|} \int_{\partial K^i_j }  (0, \chi_i \widetilde{\Phi}_{x_2, k}^{i,\text{off}} ) \cdot n_i \, ds \quad \text{in } K^i_j,
\end{equation}
with $\Phi_{x_2, k}^{i,\text{off}}  =    (0, \chi_i \widetilde{\Phi}_{x_2, k}^{i,\text{off}} ), \text{ on } \partial K^i_j$.
We write that $ \Phi_{x_1, k}^{i,\text{off}} = \mathcal{H}(\chi_i \widetilde{\Phi}_{x_1, k}^{i,\text{off}} )$ and $ \Phi_{x_2, k}^{i,\text{off}} = \mathcal{H}(\chi_i \widetilde{\Phi}_{x_2, k}^{i,\text{off}} )$, where $\mathcal{H}(v)$ is the {\it Stokes extension} of the function $v$.

Combining them, we obtain the global offline space:
\[
V_{\text{off}}  = \text{span} \{
\Phi_{x_1, k}^{i,\text{off}} \text{ and } \Phi_{x_2, k}^{i,\text{off}}: \quad
1 \leq i \leq (N_e+N_v)  \text{  and  }   1 \leq k \leq M_i
\}.
\]
Using a single index notation, we can write
\[
V_{\text{off}} = \text{span} \{ \Phi^{\text{off}}_{i} \}_{i=1}^{N_u},
\]
where $N_u =\sum_{i=1}^{N_e+N_v} M_i$.
This space will be used as the approximation space for the velocity.
For coarse approximation of pressure, we will take $Q_{\text{off}}$ to be the space of piecewise constant functions
on the coarse mesh.

\subsubsection{Online Adaptive Method}

Similar to Section \ref{sec:online_algorithm}, we will define the online velocity basis for Stokes problem.
For each coarse grid neighborhood $\omega_i^{\eps}$, we define the residual $R_i$ as a linear functional on $V^i$ such that
\begin{equation}\label{stokes-resid}
R_i(v) =  \int_{\omega_i^{\eps}}  f\cdot v \, dx -
\int_{\omega_i^{\eps}} \nabla u^m_{\text{ms}} : \nabla v dx +
\int_{\omega_i^{\eps}} p^m_{\text{ms}} \div(v) dx, \quad \forall v \in V^i
\end{equation}
where $(u^m_{\text{ms}}, p^m_{\text{ms}})$ is the multiscale solution at the enrichment level $m$, and $V^i = (H^1_0(\omega_i^{\eps}))^2$.
The norm of $R_i$ is defined as
\begin{equation}
\label{eq:Rn1}
||R_i||_{(V^i)^*} = \sup_{v \in V^i } \frac{|R_i(v)|}{\|v\|_{H^1(\omega_i^{\eps})}}.
\end{equation}
%where $V^i = H^1_0(\omega_i^{\eps})^2$.
We will then use indicators (\ref{itm:ind1}) and (\ref{itm:ind2}) for our adaptive enrichment method.
For the computation of online basis $\phi_i^{\text{on}} \in V^i_{h,0}$, we solve the following problem
\begin{equation}
\label{online_Stokes}
\begin{split}
\int_{\omega_i^{\eps}} \nabla \phi_i^{\text{on}}: \nabla v dx -
\int_{\omega_i^{\eps}} p^{\text{on}} \div(v)  dx &= R_i(v), \quad \forall v \in V^i_{h,0}, \\
 \int_{\omega_i^{\eps}} \div(\phi_i^{\text{on}}) \, q \, dx &= 0, \quad \forall q \in Q_{\text{off}}.
\end{split}
\end{equation}
%and then take $||R_i||_{V^i_h *} = ||\phi_i||_{V^i_h}$.
The adaptivity procedure follows the one presented in Section \ref{sec:online_algorithm}.

%For the construction of the adaptive online basis functions, we first choose $0<\theta<1$, for each non-oversampled coarse grid neighborhood $\omega_i^{\eps}$, find the online basis $\phi_i \in {V}^i_h$ using equation \eqref{online_Stokes}.

%Following the adaptive online algorithm in Section \ref{sec:online_algorithm}, we will perform the iterative procedure until the global error indicator is small or we have a certain number of basis functions.

\subsection{Randomized snapshots}\label{random}

In the above construction, the local problems are solved for every bounday node. This procedure is expensive and may not be 
practical. However, one can use the idea of randomized snapshots (as in \cite{randomized2014}) and reduce the
cost substantially. In randomized snapshots, one computes a few more snapshots compared to the required number of
multiscale basis functions. E.g., we compute $n+4$ snapshots for $n$ multiscale basis functions.
%One can also construct the snapshot space in a more efficient way. Proposed in \cite{randomized2014}, the randomized snapshot requires much fewer calculations but maintain a good accuracy compared with the standard snapshot space. 

To be more specific, we first generate inexpensive snapshots using random boundary conditions. Instead of solving the local problem \eqref{hrmext-elas} and \eqref{eq:hrmext} for each fine boundary degree of freedom, we solve a small number of local problems with boundary conditions:
\begin{align*}
 u_k^{+, i} &=  (r_l^i, 0) \quad \text{or} \quad (0, r_l^i) \quad \text{on} \quad \partial \omega_i^{+, \eps} \backslash  \partial \mathcal{B}^{\eps},\\
 u_k^{+, i} &=  (0, 0)  \quad \text{on} \quad  \partial \mathcal{B}^{\eps}.
\end{align*}

Here $ r_l^i$ are independent identically distributed (i.i.d.) standard Gaussian random vectors defined on the fine degree freedom of the boundary. Notice that we will solve for $u_k^{+, i}$ in a larger domain, the oversampling domian $\omega_i^{+, \eps}$. The oversampling technique is used avoid the effects of randomized boundaries. 
%We will then perform POD to achieve a lower dimensional approximation space. Finally, 
After removing dependence, we finally get our snapshot basis by taking the restriction of  $u_k^{+, i}$ in $\omega_i^{\eps}$, i.e, $u_k^i= u_k^{+, i}|_{\omega_i^\eps}$.

In Section \ref{numerical}, we will take the Stokes problem as an example and show the numerical results for randomized sanpshots. 

\section{Numerical results}
\label{numerical}

In this section, we show simulation
results using the framework of online adaptive GMsFEM presented
in Section \ref{GMsFEM}
for elasticity equations and Stokes equations. We set $\Omega = [0,1] \times [0,1]$ and use two types of perforated domains as illustrated in Figure \ref{fig:twodomain}, where the perforated regions $\mathcal {B}^{\eps}$ are circular. We have also used perforated regions of other shapes instead and obtained similar results.  The computational domain is discretized coarsely using uniform triangulation, where the coarse mesh size $H=\frac{1}{10}$ for elasticity problem and $H=\frac{1}{5}$ for Stokes problem. Furthermore, nonuniform triangulation is used inside each coarse triangular element to obtain a finer discretization. Examples of this triangulation are displayed also in Figure \ref{fig:twodomain}.

First we will choose a fixed number of offline basis (initial basis) for every coarse neighborhood, and obtain corresponding offline space $V_\text{off}$, which is also denoted by $V_{\text{ms}}^1$. Then, we perform the online iterations on non-overlapping coarse neighborhoods to obtain enriched space $V_{\text{ms}}^m$, $m\geq 1$. We will add online basis both with adaptivity and without adaptivity and compare the results. All the errors are in percentage. We note that our approaches are designed to
explore the sparsity and the adaptivity 
in the solution space and our main emphasis is on the construction
of coarse spaces. Our numerical results will show the approximation of
the fine-scale solution for different dimensional coarse spaces.

\begin{figure}
\centering
\includegraphics[width=0.8\linewidth]{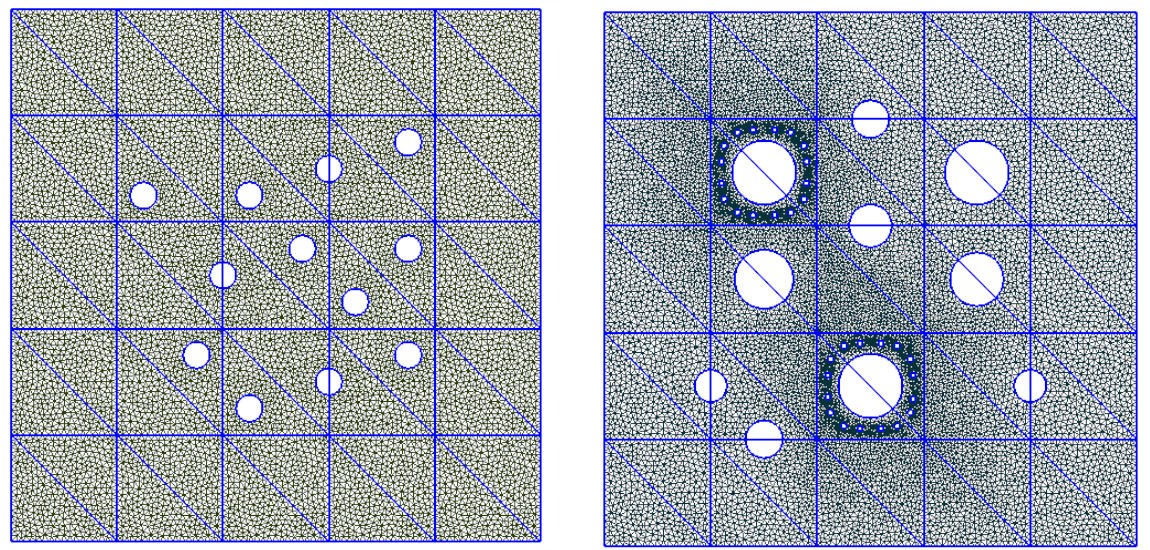}
\caption{Two heterogeneous perforated medium used in the simulations.}
\label{fig:twodomain}
\end{figure}

\subsection{Elasticity equations in perforated domain}

We consider the elasticity operator \eqref{eqn:elasticity}.
We use zero displacements $u =0$ on the inclusions,
$u_1 = 0, \sigma_2 = 0$ on the left boundary,
$\sigma_1 = 0, u_2 = 0$  on the bottom boundary and
$\sigma_1 = 0, \sigma_2 = 0$ on the right and top boundaries.
Here, $u = (u_1, u_2) $ and $\sigma = (\sigma_1, \sigma_2) $.
The source term is defined by $f = (10^7, 10^7)$,
the elastic modulus is given by $E = 10^9$, Poisson's ratio is $\nu = 0.22$, where
\[
\mu = \frac{E}{2 (1 + \nu)}, \quad
\xi = \frac{E \nu}{(1+ \nu) ( 1- 2 \nu)}.
\]
We use the following error quantities to measure the performance of the online adaptive GMsFEM
\[
||e||_{L^2} = \norm{e_u}_{L^2(\om^{\eps})} =
\frac{ \norm{(\xi +2\mu)(u - u_{\text{ms}})}_{L^2(\om^{\eps})} } {\norm{(\xi +2\mu)u}_{L^2(\om^{\eps})}} , \quad
||e||_{H^1} = \norm{e_u}_{H^1(\om^{\eps})} =
\sqrt{ \frac{\avrg{\mathcal {L}^{\eps}(u - u_{\text{ms}}),u - u_{\text{ms}}}{\om^{\eps}}} {\avrg{\mathcal {L}^{\eps}(u),u}{\om^{\eps}}}  },
\]
where $u$ and $u_{\text{ms}}$ are the fine and coarse solutions, respectively, and
$\avrg{\mathcal {L}^{\eps}(u),v}{\om^{\eps}}
=2\mu\avrg{\strain(u),\strain(v)}{\om^{\eps}}+\xi \avrg{\nabla\cdot u,\nabla\cdot v}{\om^{\eps}}$. Note that the reference solution $u$ needs a full fine scale computation. The fine grid DOF is 13262 for the domain with small perforations(left in Figure \ref{fig:twodomain}) and 21986 for the domain with big perforations (right in Figure \ref{fig:twodomain}). 

The fine-scale solution and coarse-scale solution corresponding to the two different perforated domains in Figure \ref{fig:twodomain} are presented in Figures \ref{fig:el1} and \ref{fig:el2}. Fine solutions are shown on the left of the figure, coarse offline solutions are in the middle and online solutions are on the right. In Tables \ref{tab:el_sh} and \ref{tab:el_ex}, we present the convergence history when the problem is solved in two different perforated domain with one, two and four initial bases in the left, middle and right column, respectively. Each column shows the error behavior when the online method is applied without adaptivity, with adaptivity using Indicator 1 (see \eqref{itm:ind1}) and with adaptivity using Indicator 2 (see \eqref{itm:ind2}).

Numerical results for the first perforated domain are displayed in Figure \ref{fig:el1}.  We  observe that the offline solution is close to the fine-scale solution; however, there are some missing features
in the offline solution. For example, the low values of the solution for a
connected regions around circular inclusions, while this is not the
case for the fine-scale solution. Also, we observe that the offline solution
does not capture the low values of the solution near the inclusions. On the other
hand, the solution using the online procedure with approximately the same
number of degrees of freedom as the offline solution has very good accuracy.
From Table \ref{tab:el_sh}, we observe that when using one initial basis, the $L^2$ and energy error reduce to $1.3$\% and $5.82$\% respectively after one online iteration in the case without adaptivity. However, if we select two initial bases, the the $L^2$ and energy error can be reduced to $0.567$\% and $2.92$\% respectively after one online iteration, which is almost half of the errors for one initial basis situation. When the number of basis is fixed, it shows that adding online basis can reduce the error more effectively than adding offline basis. For example, when we use two offline basis and two online basis, the energy error is $0.369$\%; while when we select four offline basis, the energy error is $26.703$\%.
Comparison of the error behavior between solving with and without adaptivity in this table shows that, error is smaller under the similar DOF when adaptive online method is applied. For example, if we start with one initial basis, the energy error is 5.482\% with DOF 500 when online method is applied without adaptivity, but the energy error becomes 2.589\% with DOF 536 when online adaptive method is applied.
When we solve with the adaptivity, we observe that the first indicator (see \eqref{itm:ind1}) is more effective when one initial basis is selected. However, if we start with two or four initial bases, the second indicator (see \eqref{itm:ind2}) gives us slightly better results.  The smallest eigenvalues are $\Lambda_{min} = 31.4, 79.9, 204.8$ when one, two and four initial basis are used.

In Figure \ref{fig:el2}, we test with a different perforated domain where the circular inclusions are larger compared to the domain in Figure \ref{fig:el1} and extremely small inclusions are set around some big ones. Comparing the offline and fine solution, we notice that some features of solution in the interior of the domain are missing, and the errors around the boundary are large. However, the online solution fix these problems well and show much better accuracy. Looking at Table \ref{tab:el_ex}, we observe that as we select more initial basis, the error decreases faster after one online iteration. For example, when one online iteration is applied without adaptivity, the $H^1$ error reduces $8.5$ times if we use one initial basis, yet it reduces around $12$ times if we use two initial basis. Considering the convergence behavior of online method with adaptivity against the online method without adaptivity, we see that the adaptivity is important. For instance, in a similar DOF of $1300$ in the case of four initial basis used, the $H^1$ error $10^{-5}$ without adaptivity, while it is only $10^{-6}$ with adaptivity.

\begin{table}[!htb]
\centering
  \begin{tabular}{ |c | c | c | }
    \hline
$DOF$& \multirow{2}{*}{$||e||_{L^2}$} &\multirow{2}{*}{$||e||_{H^1}$} \\
(\# iter)  &	& \\  \hline
\multicolumn{3}{|c|}{without adaptivity} \\ \hline
338  		&	29.269	&	 53.691	\\  \hline
500 (1)  	&	 1.300	&	  5.482	\\  \hline
662 (2)  	&	 0.082	&	  0.450	\\  \hline
824 (3)  	&	 0.010	&	  0.069	\\  \hline
986 (4)  	&	 0.0009&	  0.007	\\  \hline
\multicolumn{3}{|c|}{with adaptivity, $\eta^2_i = r^2_i$} \\ \hline
338  		&	29.269	&	 53.691	\\  \hline
510 (3)  	&	 0.567	&	  3.115	\\  \hline
654 (6)  	&	 0.042	&	  0.306	\\  \hline
852 (10)  	&	 0.001	&	  0.013	\\  \hline
1014 (13) &	 0.0001&	  0.0008	\\  \hline
\multicolumn{3}{|c|}{with adaptivity, $\eta^2_i = r^2_i \lambda_{i+1}^{-1}$} \\ \hline
338  		&	29.269	&	 53.691	\\  \hline
536 (4)  	&	 0.474	&	  2.589	\\  \hline
684 (7)  	&	 0.039	&	  0.285	\\  \hline
846 (10)  	&	 0.003	&	  0.023	\\  \hline
1002 (13) &	 0.0002&	  0.001	\\  \hline
\end{tabular}
$\;\;\;$
%\caption{Elasticity problem. One offline basis ($\theta = 0.7$). Left: online without adaptivity. Middle:with adaptivity, $\eta^2_i = r^2_i$. Right:with adaptivity, $\eta^2_i = r^2_i \lambda_{i+1}^{-1}$. $\min_i \lambda_i^{M+1} = 31.3961$}
%\end{table}
%
%%\subsubsection{Two offline basis}
%
%\begin{table}[!htb]\label{tab:el_sh_2}
%\centering
  \begin{tabular}{ |c | c | c | } \hline
$DOF$& \multirow{2}{*}{$||e||_{L^2}$} &\multirow{2}{*}{$||e||_{H^1}$} \\
(\# iter)  &	& \\  \hline
\multicolumn{3}{|c|}{without adaptivity} \\ \hline
412  		&	10.652	&	 32.862	\\  \hline
574 (1)  	&	 0.567	&	  2.921	\\  \hline
736 (2)  	&	 0.049	&	  0.369	\\  \hline
898 (3)  	&	 0.005	&	  0.047	\\  \hline
1060 (4)  	&	 0.0005&	  0.004	\\  \hline
\multicolumn{3}{|c|}{with adaptivity, $\eta^2_i = r^2_i$} \\ \hline
412  		&	10.652	&	 32.862	\\  \hline
584 (3)  	&	 0.416			&	  2.285	\\  \hline
740 (6)  	&	 0.029			&	  0.236	\\  \hline
932 (10)  	&	 0.001			&	  0.009	\\  \hline
1190 (15)	&	 1.685e-05	&	  0.0001	\\  \hline
\multicolumn{3}{|c|}{with adaptivity, $\eta^2_i = r^2_i \lambda_{i+1}^{-1}$} \\ \hline
412  		&	10.652	&	 32.862	\\  \hline
570 (3)  	&	 0.437			&	  2.519	\\  \hline
730 (6)  	&	 0.031			&	  0.252	\\  \hline
924 (10)  	&	 0.001			&	  0.009	\\  \hline
1072 (13)	&	 8.772e-05	&	  0.0006	\\  \hline
  \end{tabular}
  $\;\;\;$
%\caption{Elasticity problem. Two offline basis ($\theta = 0.7$). Left: online without adaptivity. Middle:with adaptivity, $\eta^2_i = r^2_i$. Right:with adaptivity, $\eta^2_i = r^2_i \lambda_{i+1}^{-1}$. $\min_i \lambda_i^{M+1} = 79.8957$}
%\end{table}
%
%%\subsubsection{Four offline basis}
%
%\begin{table}[!htb]\label{tab:el_sh_4}
%\centering
  \begin{tabular}{ | c | c | c | }    \hline
$DOF$& \multirow{2}{*}{$||e||_{L^2}$} &\multirow{2}{*}{$||e||_{H^1}$} \\
(\# iter)  &	& \\  \hline
\multicolumn{3}{|c|}{without adaptivity} \\ \hline
648  		&	7.414	&	 26.703	\\  \hline
810 (1)  	&	 0.479	&	  2.509	\\  \hline
972 (2)  	&	 0.046	&	  0.368	\\  \hline
1134 (3)  	&	 0.004	&	  0.043	\\  \hline
1296 (4)  	&	 0.0005&	  0.004	\\  \hline
\multicolumn{3}{|c|}{with adaptivity, $\eta^2_i = r^2_i$} \\ \hline
648  		&	7.414	&	 26.703	\\  \hline
808 (3)  	&	 0.303	&	  1.977	\\  \hline
980 (6)  	&	 0.022	&	  0.192	\\  \hline
1144 (9)  	&	 0.001	&	  0.016	\\  \hline
1302 (12)	&	 0.0001&	  0.001	\\  \hline
\multicolumn{3}{|c|}{with adaptivity, $\eta^2_i = r^2_i \lambda_{i+1}^{-1}$} \\ \hline
648  		&	7.414	&	 26.703	\\  \hline
808 (3)  	&	 0.300			&	 1.776	\\  \hline
976 (6)  	&	 0.019			&	  0.173	\\  \hline
1174 (10)	&	 0.0006		&	  0.005	\\  \hline
1338 (13)	&	 3.492e-05	&	  0.0002	\\  \hline
  \end{tabular}
\caption{Elasticity problem in the perforated domain with small inclusions (Figure \ref{fig:twodomain}, left). One (Left), Two (Middle) and Four (Right) offline basis functions ($\theta = 0.7$). } \label{tab:el_sh}
\end{table}

\begin{figure}[!htb]
\centering
    \includegraphics[width=1.0\textwidth]{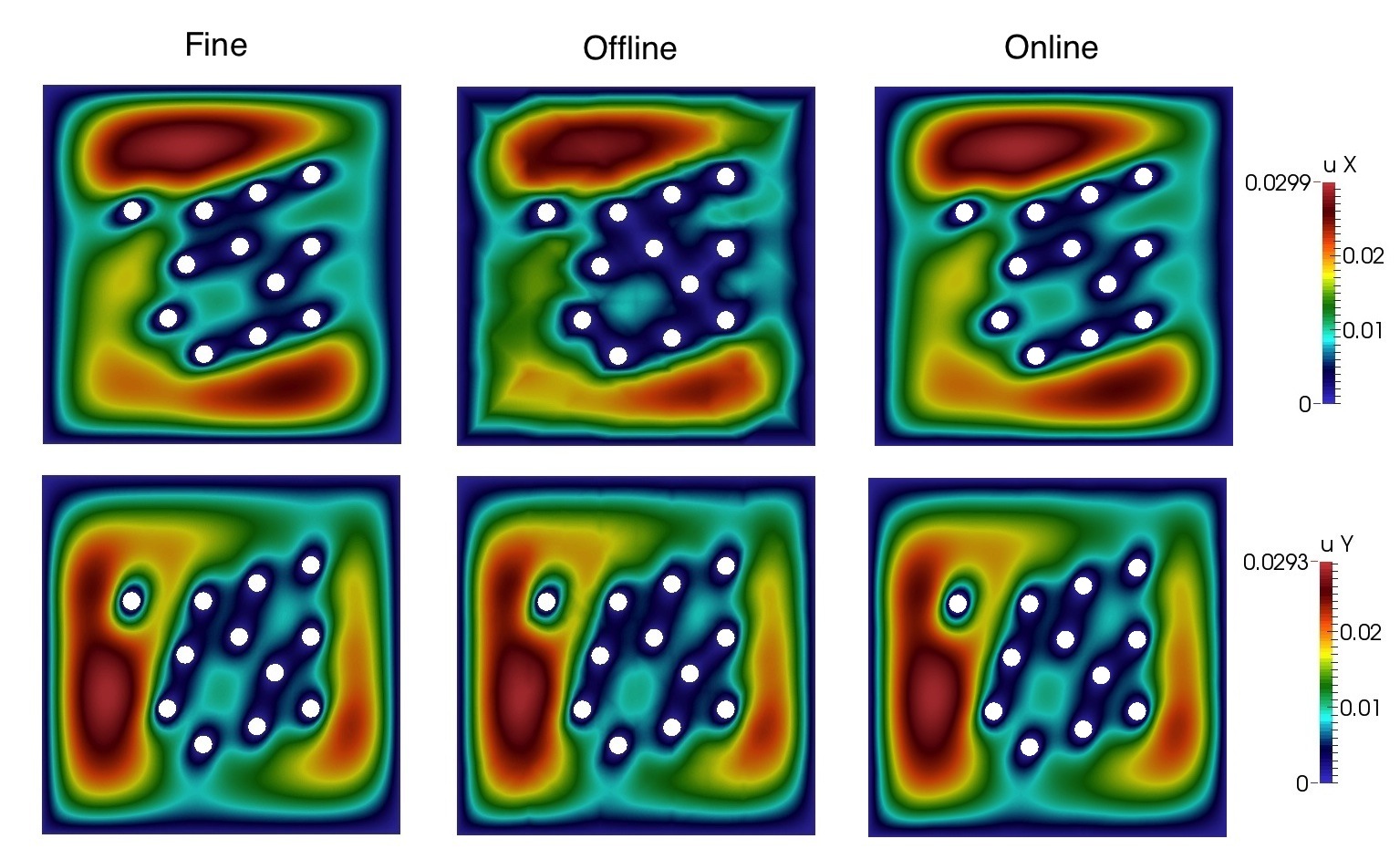}
    \caption{Elasticity problem in the perforated domain with small inclusions (Figure \ref{fig:twodomain}, left). Comparison of solutions in: Fine scale (left) $DOF = 13262$, Coarse-scale offline, $DOF=412$ (middle), Coarse-scale online without adaptivity, $DOF=574$ (right). Top: $u_1$. Bottom: $u_2$. }
    \label{fig:el1}
\end{figure}

% EXTREME CASE

%\subsubsection{One offline basis}
%COARSE: cells=200, facets = 320, nodes=121
%FINE: cells=21232, facets = 32266, nodes=10993
%Fine scale = 21986

\begin{table}[!htb]
\centering
  \begin{tabular}{ |c | c | c | } \hline
$DOF$& \multirow{2}{*}{$||e||_{L^2}$} &\multirow{2}{*}{$||e||_{H^1}$} \\
(\# iter)  &	& \\  \hline
\multicolumn{3}{|c|}{without adaptivity} \\ \hline
278  		&	38.074	&	 61.168 \\ \hline
440 (1)  	&	 2.098	&	  7.181 \\ \hline
602 (2)  	&	 0.167	&	  0.670 \\ \hline
764 (3)  	&	 0.021	&	  0.114 \\ \hline
926 (4)  	&	 0.001	&	  0.010 \\ \hline
\multicolumn{3}{|c|}{with adaptivity, $\eta^2_i = r^2_i$} \\ \hline
278  		&	38.074	&	 61.168 \\ \hline
436 (3)  	&	 1.058	&	  4.493 \\ \hline
628 (7)  	&	 0.029	&	  0.175 \\ \hline
760 (10)  	&	 0.002	&	  0.014 \\ \hline
950 (14)  	&	 5.339e-05	&	  0.0003 \\ \hline
\multicolumn{3}{|c|}{with adaptivity, $\eta^2_i = r^2_i \lambda_{i+1}^{-1}$} \\ \hline
278  		&	38.074	&	 61.168 \\ \hline
436 (3)  	&	 1.733	&	  7.005 \\ \hline
614 (7)  	&	 0.074	&	  0.399 \\ \hline
748 (10)  	&	 0.005	&	  0.037 \\ \hline
940 (14)  	&	 0.0002&	  0.001 \\ \hline
  \end{tabular}
  $\;\;\;$
%\caption{Elasticity problem. One offline basis ($\theta = 0.7$). Left: online without adaptivity. Middle:with adaptivity, $\eta^2_i = r^2_i$. Right:with adaptivity, $\eta^2_i = r^2_i \lambda_{i+1}^{-1}$.  $\min_i \lambda_i^{M+1} = 31.3961$}
%\end{table}
%
%%\subsubsection{Two offline basis}
%
%\begin{table}[!htb]\label{tab:el_ex_2}
%\centering
  \begin{tabular}{ | c | c | c | } \hline
$DOF$& \multirow{2}{*}{$||e||_{L^2}$} &\multirow{2}{*}{$||e||_{H^1}$} \\
(\# iter)  &	& \\  \hline
\multicolumn{3}{|c|}{without adaptivity} \\ \hline
382  		&	15.585	&	 38.387 \\ \hline
544 (1)  	&	 0.794	&	  3.239 \\ \hline
706 (2)  	&	 0.071	&	  0.397 \\ \hline
868 (3)  	&	 0.008	&	  0.054 \\ \hline
1030 (4)  	&	 0.0006&	  0.003 \\ \hline
\multicolumn{3}{|c|}{with adaptivity, $\eta^2_i = r^2_i$} \\ \hline
382  		&	15.585	&	 38.387 \\ \hline
556 (3)  	&	 0.477	&	  2.116 \\ \hline
704 (6)  	&	 0.033	&	  0.211 \\ \hline
892 (10)  	&	 0.001	&	  0.007 \\ \hline
1038 (13)	&	 8.760e-05	&	  0.0005 \\ \hline
\multicolumn{3}{|c|}{with adaptivity, $\eta^2_i = r^2_i \lambda_{i+1}^{-1}$} \\ \hline
382  		&	15.585	&	 38.387 \\ \hline
548 (3)  	&	 0.528	&	  2.377 \\ \hline
740 (7)  	&	 0.019	&	  0.124 \\ \hline
878 (10)  	&	 0.001	&	  0.010 \\ \hline
1064 (14)	&	 4.710e-05	&	  0.0003 \\ \hline
  \end{tabular}
  $\;\;\;$
%  \caption{Elasticity problem. Two offline basis ($\theta = 0.7$). Left: online without adaptivity. Middle:with adaptivity, $\eta^2_i = r^2_i$. Right:with adaptivity, $\eta^2_i = r^2_i \lambda_{i+1}^{-1}$. $\min_i \lambda_i^{M+1} = 79.8957$}
%\end{table}
%
%%\subsubsection{Four offline basis}
%
%\begin{table}[!htb]\label{tab:el_ex_4}
%\centering
  \begin{tabular}{ | c | c | c | }  \hline
$DOF$& \multirow{2}{*}{$||e||_{L^2}$} &\multirow{2}{*}{$||e||_{H^1}$} \\
(\# iter)  &	& \\  \hline
\multicolumn{3}{|c|}{without adaptivity} \\ \hline
648  		&	8.870	&	 27.343 \\ \hline
810 (1)    	&	0.611  	&	 2.390 \\ \hline
972 (2)    	&	0.063  	&	 0.376 \\ \hline
1134 (3)  	&	0.006  	&	 0.042 \\ \hline
1296 (4)  	&	0.0005 &	 0.003 \\ \hline
\multicolumn{3}{|c|}{with adaptivity, $\eta^2_i = r^2_i$} \\ \hline
648  		&	8.870	&	 27.343 \\ \hline
820 (3)    		&	0.301  			&	 1.400 \\ \hline
972 (6)    		&	0.021  			&	 0.140 \\ \hline
1154 (10)  	&	0.0006  		&	 0.004 \\ \hline
1300 (13)  	&	3.784e-05  	&	 0.0002 \\ \hline
\multicolumn{3}{|c|}{with adaptivity, $\eta^2_i = r^2_i \lambda_{i+1}^{-1}$} \\ \hline
648  		&	8.870	&	 27.343 \\ \hline
810 (3)    		&	0.309  			&	 1.500 \\ \hline
996 (7)    		&	0.008  			&	 0.067 \\ \hline
1138 (10)  	&	0.0006  		&	 0.005 \\ \hline
1314 (14)  	&	1.659e-05  	&	 0.0001 \\ \hline
  \end{tabular}
\caption{Elasticity problem in the perforated domain with big inclusions (Figure \ref{fig:twodomain}, right). One (Left), Two (Middle) and Four (Right) offline basis functions ($\theta = 0.7$). }\label{tab:el_ex}
\end{table}

\newpage
\begin{figure}[!htb]
\centering
    \includegraphics[width=1.0\textwidth]{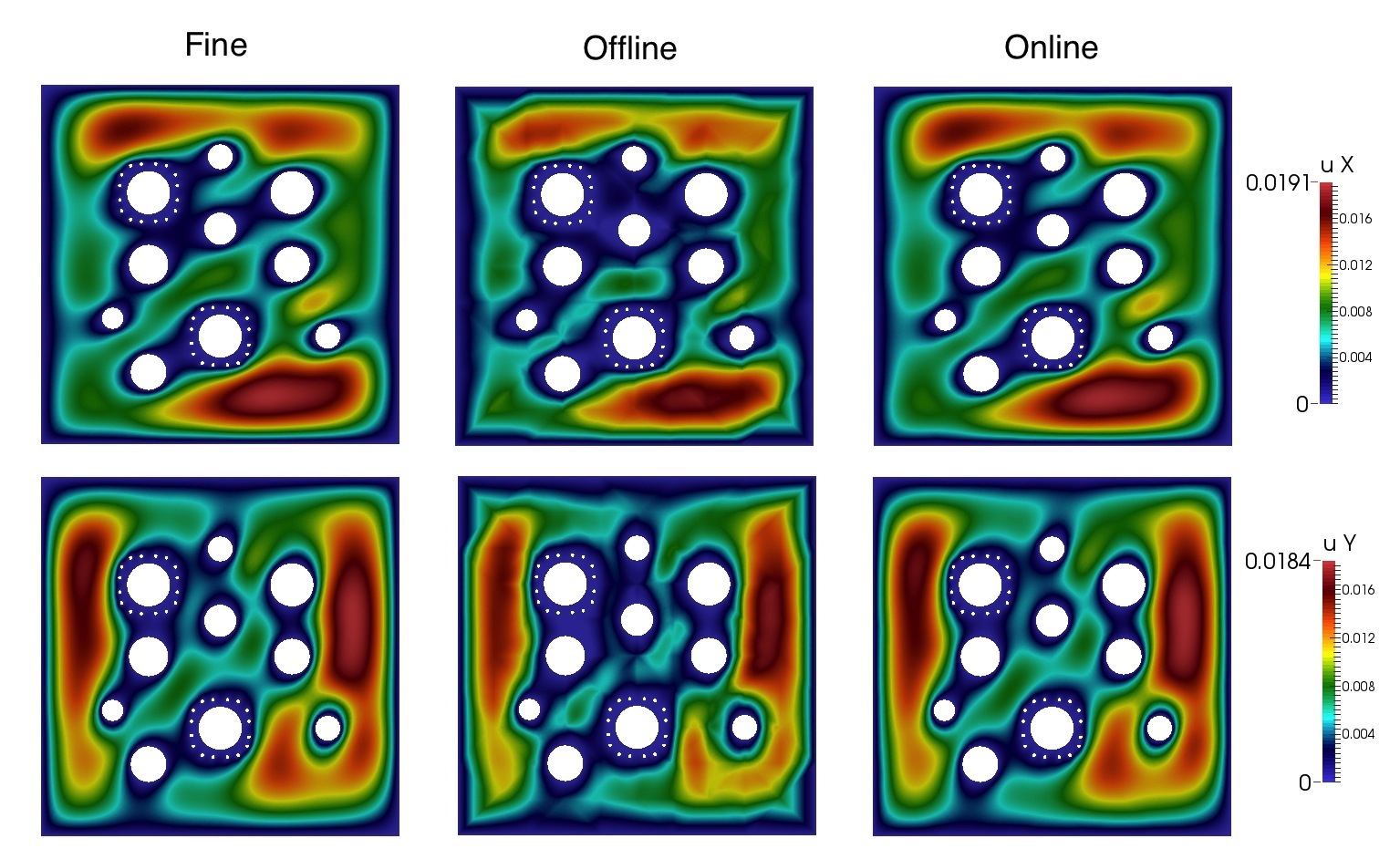}
    \caption{Elasticity problem  in the perforated domain with big inclusions (Figure \ref{fig:twodomain}, right). Comparison of solutions in: Fine scale (left) $DOF = 21986$, Coarse scale offline, $DOF=382$ (middle), Coarse scale online without adaptivity, $DOF=544$ (right). Top: $u_1$. Bottom: $u_2$. }
    \label{fig:el2}
\end{figure}

\subsection{Stokes equations in perforated domain}

In our final example, we consider the Stokes operator \eqref{eqn:Stokes}
with zero velocity $u = (0, 0)$ on $\partial \om^{\eps}\cap  \partial \mathcal {B}^{\eps}$ and $\frac{\partial{u}}{\partial{n}} = (0, 0)$ on $\partial \om$, and source term $f = (1, 1) \in \Omega^{\eps}$.
For the fine-scale approximation of the Stokes problem, we use $P_2$ elements for velocity and piecewise constants for pressure.
To improve the accuracy of multiscale solutions, we have enriched velocity spaces by adding online velocity basis.

The errors will be measured in relative $L^2$ and $H^1$norms for velocity and $L^2$ norm for pressure
\begin{align*}
||e_u||_{L^2} = \norm{e_u}_{L^2(\Omega^{\epsilon})} &=
\frac{ \norm{u-u_{\text{ms}}}_{L^2(\Omega^{\epsilon})} }{\norm{u}_{L^2(\Omega^{\epsilon})}} , \quad
||e_u||_{H^1} = \norm{e_u}_{H^1(\Omega^{\epsilon})} =
\frac{ \norm{u-u_{\text{ms}}}_{H^1(\Omega^{\epsilon})} }{\norm{u}_{H^1(\Omega^{\epsilon})}} , \\
||e_p||_{L^2(\Omega^{\epsilon})} &= \frac{ \norm{\bar{p}-p_{\text{ms}}}_{L^2(\Omega^{\epsilon})} }{\norm{\bar{p}}_{L^2(\Omega^{\epsilon})}},
\end{align*}
where $(u, p)$ and $(u_{\text{ms}},p_{\text{ms}})$ are  fine-scale and coarse-scale solutions, respectively for velocity and pressure, and $\bar{p}$ is the cell average of the fine scale pressure, that is,
$\bar{p} = \frac{1}{|K_i^{\epsilon}|} \int_{K_i^{\epsilon}} p $ for all $K_i^{\epsilon} \in {\cal T}^{H}$. Notice that we solve the reference solution $(u, p)$ on a full fine grid. The fine grid DOF is 77524 for the domain with small perforations(left in Figure \ref{fig:twodomain}) and 101386 for the domain with big perforations (right in Figure \ref{fig:twodomain}).

\subsubsection{Randomized snapshots}
As mentioned in Section \ref{random}, we will show the numerical results of Stokes problem for the offline GMsFEM using randomized snapshots. The convergence behaviors are shown in Tables \ref{rand-sh} and \ref{rand-ex} for perforated domain with small inclusions (Figure \ref{fig:twodomain}, left) and big inclusions (Figure \ref{fig:twodomain}, right), respectively,
where the notation $\omega^{+, \eps} = \omega^{\eps} + 4$ means that the oversampled region $\omega^{+,\eps}$
is obtained by enlarging the region $\omega^{\eps}$ by $4$ fine grid cells.
From these tables, we observe that the approach using randomized snapshots is more efficient since much fewer snapshot functions are used to achieve comparable accuracy. In particular, we get similar errors when the number of randomized snapshots is only around $20\%$ of the number of standard snapshots. Notice that in the randomized snapshot construction, we need to add the constant basis, i.e, the constant function in each $\omega^{+, \eps}$. Note that, we do not have the constant basis in domain with inclusions when calculating snapshot basis in the standard way. This additional constant basis function makes the errors smaller for low degrees of freedom. For example, in the domain with small inclusions, when $DOF = 534$, the velocity $L^2$ error is $12.49\%$ when we use standard snapshots, while the error is only $4.54\%$ when the dimension of the randomized snapshots is $24.1\%$ of the dimension 
of the whole snapshot space
(see Table \ref{rand-sh}). However, when the $DOF$ becomes larger, the errors for randomized snapshots are similar to that for standard snapshots. For instance, the velocity $L^2$ error is $0.07\%$ when $DOF = 1986$ in domain with big inclusions for both standard snapshots and randomized snapshots (see Table \ref{rand-ex}), where
the dimension of the randomized snapshot is $13.8\%$ of the dimension of the whole snapshots.  
%On the other hand, one can use a larger ratio of randomized snapshots to get errors closer to standard snapshots. 
We remark that, by
balancing the computational cost and accuracy, we find the results are satisfactory when $24.1\%$ randomized snapshots for domain with small inclusions(Figure \ref{fig:twodomain}, left) and $20.7\%$ randomized snapshots for domain with big inclusions(Figure \ref{fig:twodomain}, right) are used. 

\begin{table}[!h]
\centering
  \begin{tabular}{ | c | c | c | c  |  }
    \hline
$DOF$  
& $||\varepsilon_u||_{L^2(\Omega)}$(\%)	
& $||\varepsilon_u||_{H^1(\Omega)}$(\%)	
& $||\varepsilon_{\bar{p} }||_{L^2(\Omega)}$(\%)	 \\  \hline  
\multicolumn{4}{|c|}{Standard snapshot (100\%)} \\
\hline
534 	&12.49  &36.91 &21.46 	\\  \hline
1018 &0.28  &4.67  & 0.86 	\\  \hline
1986 & 0.031 &1.64  &0.0029 	\\  \hline
\multicolumn{4}{|c|}{Randomized snapshot: $\omega^{+, \eps} = \omega^{\eps} + 4$, 18.1\%} \\
\hline
534 	&4.99  & 23.95  &13.4 	\\  \hline
1018 	&0.54  &7.05 &0.53 	\\  \hline
1986 	&0.04  &1.77   &0.02 	\\  \hline
\multicolumn{4}{|c|}{Randomized snapshot: $\omega^{+, \eps} = \omega^{\eps} + 4$, 24.1\%} \\
\hline
534 	& 4.54  & 22.69  &8.28 	\\  \hline
1018 	&  0.47  &6.6  & 0.52	\\  \hline
1986 	& 0.036 &1.72 &0.009 	\\  \hline
  \end{tabular}
\caption{Perforated domain with small inclusions (Figure \ref{fig:twodomain}, left) for the Stokes problem using standard snapshots and randomized snapshots.}
\label{rand-sh}
\end{table}

\begin{table}[!h]
\centering
  \begin{tabular}{  | c | c | c  | c | }
    \hline
 $DOF$  
& $||\varepsilon_u||_{L^2(\Omega)}$(\%)	
& $||\varepsilon_u||_{H^1(\Omega)}$(\%)	
& $||\varepsilon_{\bar{p} }||_{L^2(\Omega)}$(\%)	 \\  \hline  
\multicolumn{4}{|c|}{Whole snapshot (100\%)} \\
\hline
534 	&11.34 &34.49 &16.18 	\\  \hline
1018 	& 0.17 & 3.62  &1.09  	\\  \hline
1986 	& 0.07 &2.44   &0.006	\\  \hline
\multicolumn{4}{|c|}{Randomized snapshot: $\omega^{+, \eps} = \omega^{\eps} + 4$, 13.8\%} \\
\hline
534 	& 6.04 &24.84 &9.37 	\\  \hline
1018 	&0.66  &7.27  &0.95  	\\  \hline
1986 	&0.07 &2.53  & 0.02	\\  \hline
\multicolumn{4}{|c|}{Randomized snapshot: $\omega^{+, \eps} = \omega^{\eps} + 4$, 20.7\%} \\
\hline
534 	& 5.3 & 23.39   &14.95	\\  \hline
1018 	&  0.56 & 6.87  &0.73 	\\  \hline
1986 	& 0.07 &2.51 &0.015	\\  \hline
  \end{tabular}
\caption{Perforated domain with big inclusions (Figure \ref{fig:twodomain}, right) for the Stokes problem using standard snapshots and randomized snapshots.}
\label{rand-ex}
\end{table}

\subsubsection{Adaptive online results}
In this section, we present adaptive online results for Stokes problem for two perforated domains depicted in Figure \ref{fig:twodomain}. The solutions are shown in Figure \ref{fig:st-sh} and Figure \ref{fig:st-ex}. In these figures, the $x_1$-component and $x_2$-component of the velocity solution are shown in the first and second rows,  and the pressure solution is presented in the third row. The three columns contain the fine-scale, coarse-scale offline and coarse-scale online solutions. In both cases, we observe that the offline velocity solution is not able to capture the low values at the corners of the domain. Some features between inclusions also do not appear correctly in the offline solution. For example, in Figure \ref{fig:st-sh}, the low values in the upper left and lower right corner of the domain are missing in the offline velocity solution. However, it was recovered very well in the online solution. Also, compared to the fine-scale solution, the features between the first hole on the left and the other inclusions are not captured in the offline solution. However, the online solutions get these features well and outputs almost same profiles as the fine solution. In Figure \ref{fig:st-ex}, for the domain has big inclusions with some extremely small inclusions around, we see even worse behavior of the offline solution compared to that in Figure \ref{fig:st-sh}, where the domain has several small inclusions. The low values of the velocity solution in the $x_2$-component along the right boundary are almost missing in the offline solution. The offline velocity solutions in both components  around inclusions are still very poor. These observations highlights the advantage of the online method. We performed other tests for different perforated domains, and the results also suggest that online method is quite necessary.

Now, we turn our attention to velocity $L^2(\om^{\eps})$, $H^1(\om^{\eps})$  errors and pressure $L^2(\om^{\eps})$ error presented in Table \ref{tab:st-sh} and Table \ref{tab:st-ex}. We consider different numbers of initial basis on each coarse neighborhood. For the perforated domain with small inclusions in Figure \ref{fig:st-sh}, we observe from Table \ref{tab:st-sh}  that both the velocity and pressure error decrease faster as we choose more initial bases. For example, the velocity has large $H^1$ error $66.28$\% using one initial basis. After adding one online basis, it reduces to $22.3$\%. When two initial bases are selected, the  velocity $H^1$ error reduces from $23.4$\% to $3.2$\% after one step enrichment. Fixing the number of initial basis, we can compare the error behavior for the online method with or without adaptivity. It appears that online adaptive method reduces the errors more effectively. For instance, when one initial basis is selected, the velocity $H^1$ error is $22.302$\% for DOF $488$ using non-adaptive online algorithm, while it is only $3.067$\% for a similar DOF $499$ using adaptive online method with indicator 1 (see \eqref{itm:ind1}). Comparing two error indicators for adaptive online method, we see that the indicator 1 is preferred when choosing one initial basis. Since the velocity error is $8.758$\% for DOF $504$ using indicator 2 (see \eqref{itm:ind2}), which is much larger than $3.067$\%. Also, the pressure error is $16.559$\% in this case when using indicator 2, which is almost $5$ times larger compared with $3.575$\% when using indicator 1. However, both indicator works well when selecting more initial bases. We see very similar errors for both velocity and pressure fields using different indicators when the number of initial basis is two or three.

For the second example in Figure \ref{fig:st-ex}, results are shown in Table \ref{tab:st-ex}. In this case, we observe that the online approach works better if we start with more initial basis. For example, the velocity $H^1$ error is $71.823$\% with one initial velocity basis, and reduces to $24.460$\% after adding one online basis. However it's only $20.430$\% with two initial basis without online enrichment. This implies that it is better to start with two or more initial basis in order to see that the more the online basis are used, the smaller the errors become. Similarly as before, the online approach with the adaptivity reduces the errors faster. Compared the two indicators, we see that the first error indicator (see \eqref{itm:ind1}) for the adaptive online method gives slightly better results for any number of initial basis. One can also find that the pressure error also reduces significantly when we only enrich the velocity space.

% SH
\begin{figure}[!h]
\begin{center}
        \includegraphics[width=1.0\textwidth]{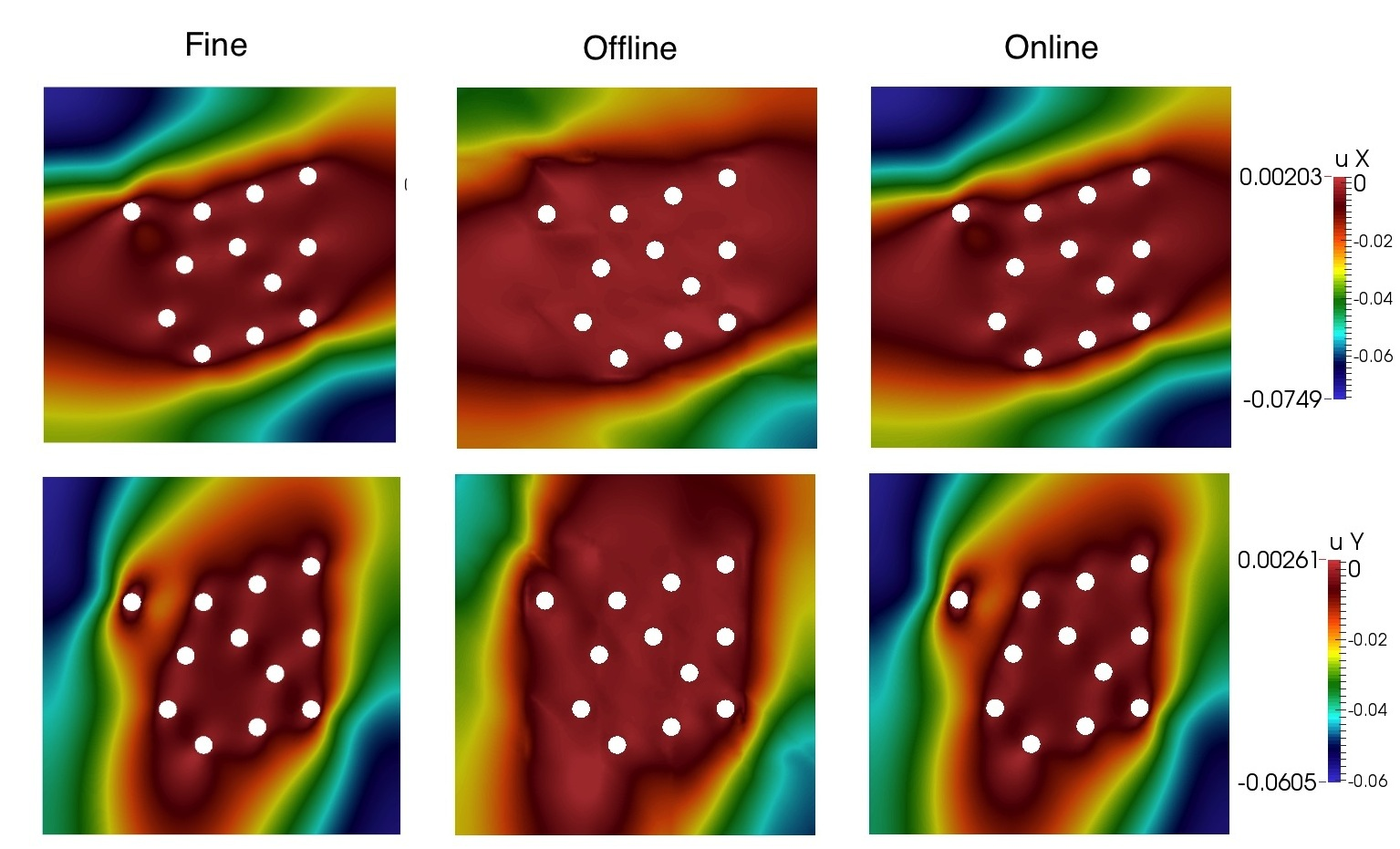}
        \includegraphics[width=1.0\textwidth]{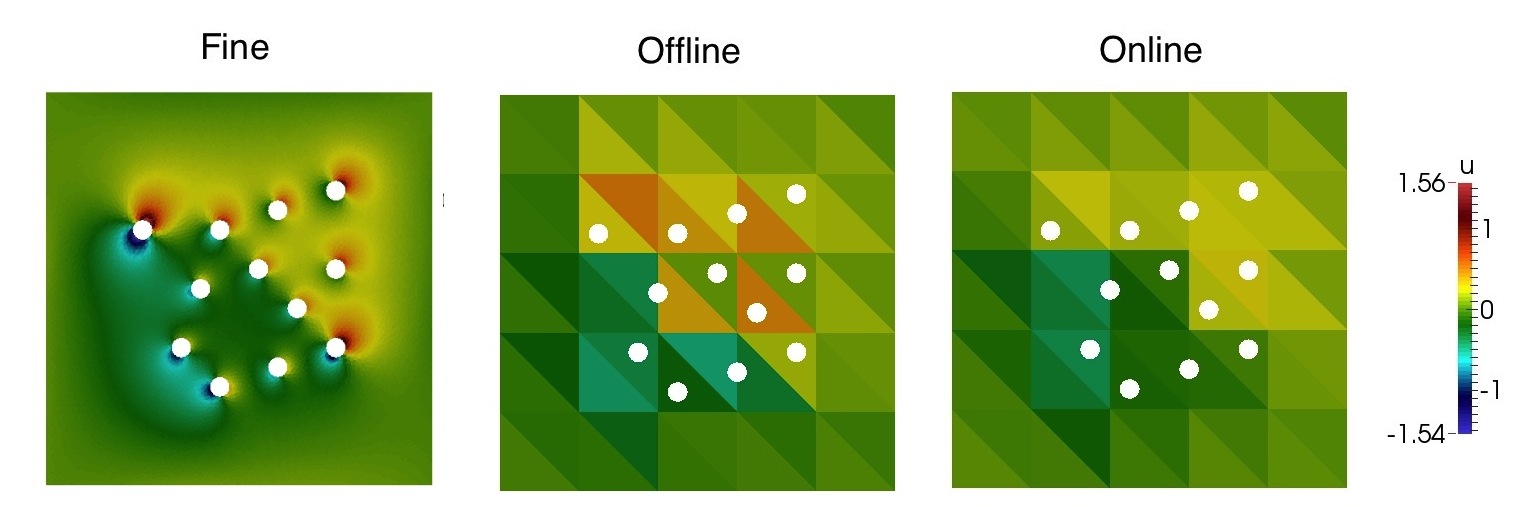}
\end{center}
    \caption{Stokes problem. Fine-scale and multiscale solutions for velocity and pressure ($u_1$ (Top), $u_2$ (Middle) and $p$ (Bottom)) in perforated domain with small inclusions(Figure \ref{fig:twodomain}, left). Left: fine-scale solution, $DOF = 77524$. Middle: multiscale solutions using 1 multiscale basis function for velocity, $DOF = 452$, velocity $L^2$ error is $42.439$ \%. Right: multiscale solutions after 2 online iteration without adaptivity, $DOF = 524$, velocity $L^2$ error is $1.688$ \%. }
    \label{fig:st-sh}
\end{figure}

\begin{table}[!h]
\centering
\begin{tabular}{ |c | c | c | c | }
     \hline
$DOF$& \multirow{2}{*}{$||e_u||_{L^2}(\%)$}
&\multirow{2}{*}{$||e_u||_{H^1}(\%)$}
& \multirow{2}{*}{$||e_{\bar{p} }||_{L^2}(\%)$} \\
(\# iter)  &	& & \\  \hline
452  		& 42.439 	& 66.276  	& 81.954	\\  \hline
\multicolumn{4}{|c|}{without adaptivity } \\ \hline
488 (1)   	& 6.300 	&  22.302 &  41.776	\\  \hline
524 (2)   	& 1.688 	&  7.191 	&  10.376	\\  \hline
596 (4)   	& 0.111 	&  0.692 	&  0.809	\\  \hline
740 (8)   	& 0.042 	&  0.514 	&  0.036	\\  \hline
\multicolumn{4}{|c|}{with adaptivity, $\eta^2_i = r^2_i$} \\ \hline
499 (4)   	& 0.627 	&  3.067 	&  3.575	\\  \hline
532 (6)   	& 0.074 	&  0.772 	&  0.329	\\  \hline
596 (10)   & 0.042 	&  0.515 	&  0.038	\\  \hline
723 (20)   &  0.033 	&  0.411 	&  0.146	\\  \hline
\multicolumn{4}{|c|}{with adaptivity, $\eta^2_i = r^2_i \lambda_{i+1}^{-1}$} \\ \hline
504 (3)    & 1.397	&  8.758	&  16.559	\\  \hline
546 (5)  	& 0.411	&  2.617	&  3.594	\\  \hline
611 (8)  	& 0.089	&  0.709	&  0.482	\\  \hline
750 (15)  	& 0.042	&  0.517	&  0.036	\\  \hline
\end{tabular}
$\;\;\;$
%\caption{One offline basis ($\theta = 0.7$) for perforated domain with small holes}
%\label{err-sh-online-1}
%\end{table}
%
%
%
%\begin{table}[!h]\label{tab:st-sh-2}
%\centering
\begin{tabular}{ |c | c | c | c | }
     \hline
$DOF$& \multirow{2}{*}{$||e_u||_{L^2}(\%)$}
&\multirow{2}{*}{$||e_u||_{H^1}(\%)$}
& \multirow{2}{*}{$||e_{\bar{p} }||_{L^2}(\%)$} \\
(\# iter)  &	& & \\  \hline
694 		& 5.467 	& 23.329  	& 13.775	\\  \hline
\multicolumn{4}{|c|}{without adaptivity } \\ \hline
730 (1)  	&  0.400 	&  3.212 	&  1.187	\\  \hline
766 (2)  	&  0.066 	&  1.137 	&  0.135	\\  \hline
838 (4)  	&  0.033 	&  0.614 	&  0.053	\\  \hline
982 (8)  	&  0.011 	&  0.216 	&  0.016	\\  \hline
\multicolumn{4}{|c|}{with adaptivity, $\eta^2_i = r^2_i$} \\ \hline
732 (3)  	&  0.093 	&  1.335 	&  0.227	\\  \hline
781 (6)  	&  0.041 	&  0.742 	&  0.079	\\  \hline
844 (10)  	&  0.019 	&  0.367 	&  0.021	\\  \hline
992 (20)  	&  0.004 	&  0.104 	&  0.003	\\  \hline
\multicolumn{4}{|c|}{with adaptivity, $\eta^2_i = r^2_i \lambda_{i+1}^{-1}$} \\ \hline
745 (2)  	& 0.088	&  1.362	&  0.310	\\  \hline
769 (3)  	& 0.057	&  0.982	&  0.110	\\  \hline
841 (6)  	& 0.030	&  0.562	&  0.026	\\  \hline
988 (12)  	& 0.011	&  0.216	&  0.022	\\  \hline
\end{tabular}
$\;\;\;$
%\caption{Two offline basis ($\theta = 0.7$) for perforated domain with small holes}
%\label{err-sh-online-2}
%\end{table}
%
%
%\begin{table}[!h]\label{tab:st-sh-4}
%\centering
 \begin{tabular}{ |c | c | c | c | }
    \hline
$DOF$& \multirow{2}{*}{$||e_u||_{L^2}(\%)$}
&\multirow{2}{*}{$||e_u||_{H^1}(\%)$}
& \multirow{2}{*}{$||e_{\bar{p} }||_{L^2}(\%)$} \\
(\# iter)  &	&&  \\  \hline
936	 	& 0.936 	& 8.795  	&  8.515	\\  \hline
\multicolumn{4}{|c|}{without adaptivity } \\ \hline
972 (1)  	&  0.032 	&  0.782 	&  0.118	\\  \hline
1008 (2)  	&  0.013 	&  0.445 	&  0.018	\\  \hline
1080 (4)  	&  0.007 	&  0.261 	&  0.005	\\  \hline
1224 (8)  	&  0.003 	&  0.106 	&  0.001	\\  \hline
\multicolumn{4}{|c|}{with adaptivity, $\eta^2_i = r^2_i$} \\ \hline
975 (3)  		&  0.016 	&  0.493 	&  0.026	\\  \hline
1011 (6)  		&  0.009 	&  0.311 	&  0.007	\\  \hline
1082 (10)  	&  0.005 	&  0.167 	&  0.003	\\  \hline
1227 (20)  	&  0.002 	&  0.078 	&  0.001	\\  \hline
\multicolumn{4}{|c|}{with adaptivity, $\eta^2_i = r^2_i \lambda_{i+1}^{-1}$} \\ \hline
1003 (2)  	& 0.013	&  0.449	&  0.018	\\  \hline
1037 (3)  	& 0.010	&  0.343	&  0.007	\\  \hline
1105 (5)  	& 0.006	&  0.218	&  0.004	\\  \hline
1241 (9)  	& 0.002	&  0.094	&  0.001	\\  \hline
\end{tabular}
\caption{Stokes problem for perforated domain with small inclusions(Figure \ref{fig:twodomain}, left). One (Upper left), Two (Upper right) and Three e(Bottom) offline basis functions ($\theta = 0.7$). }
\label{tab:st-sh}
\end{table}

% EX
\begin{figure}[!h]
\begin{center}
        \includegraphics[width=1.0\textwidth]{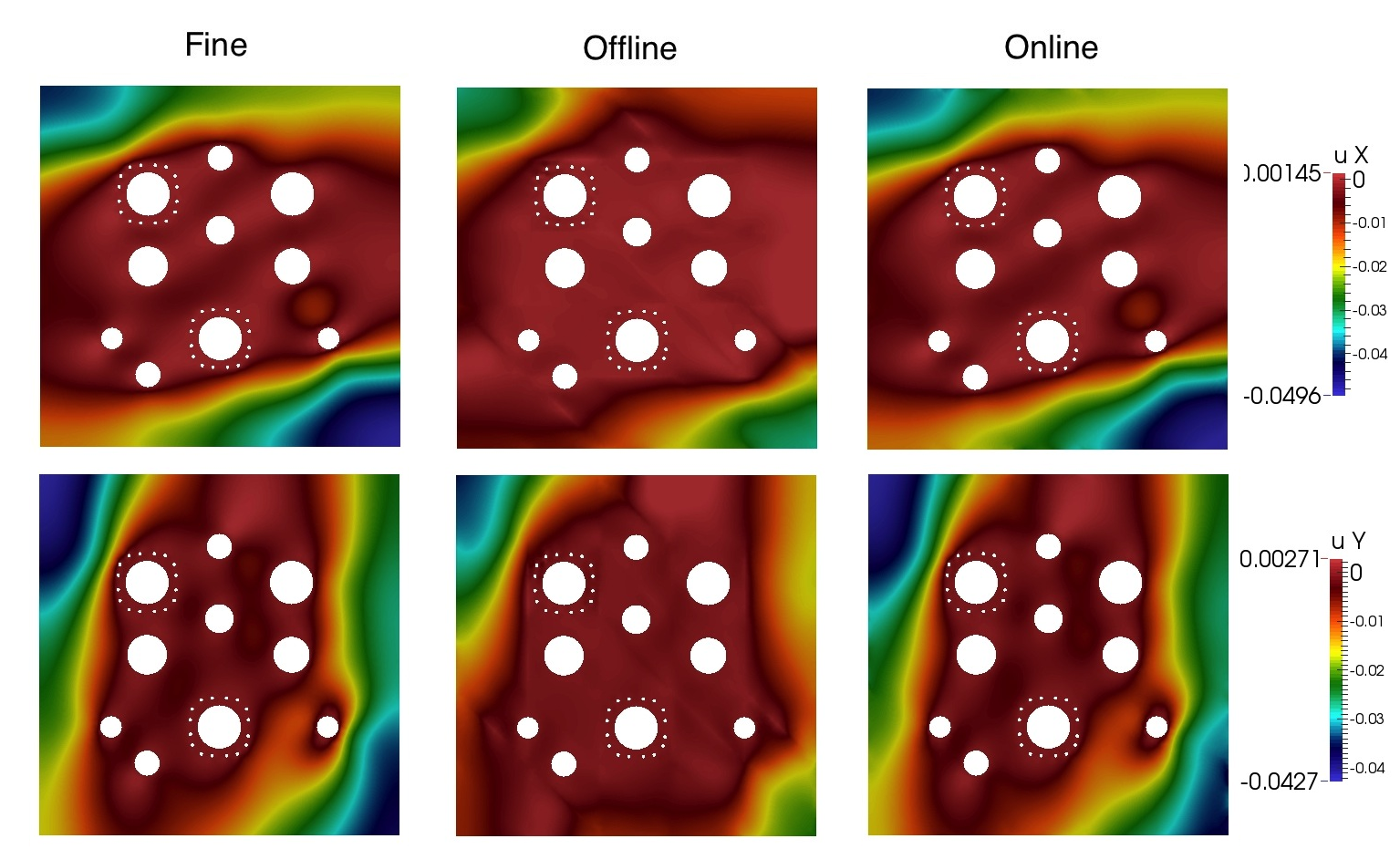}
        \includegraphics[width=1.0\textwidth]{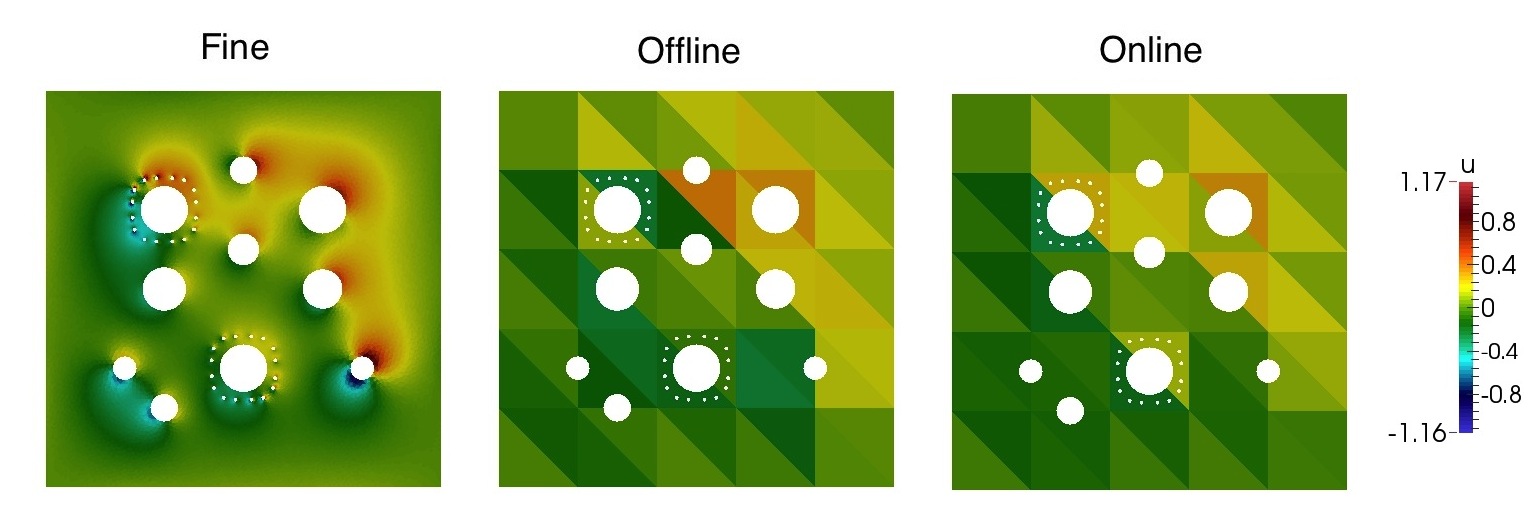}
\end{center}
    \caption{Stokes problem. Fine-scale and multiscale solutions for velocity and pressure ($u_1$ (Top), $u_2$ (Middle) and $p$ (Bottom)) in perforated domain with big inclusions(Figure \ref{fig:twodomain}, right). Left: fine-scale solution, $DOF = 101386$. Middle: multiscale solutions using 1 multiscale basis function for velocity, $DOF = 452$, velocity $L^2$ error is $47.943$ \%. Right: multiscale solutions after 2 online iteration without adaptivity, $DOF = 524$, velocity $L^2$ error is $2.266$ \%. }
    \label{fig:st-ex}
\end{figure}

\begin{table}[!h]
\centering
\begin{tabular}{ |c | c | c | c | }
     \hline
$DOF$& \multirow{2}{*}{$||e_u||_{L^2}(\%)$}
&\multirow{2}{*}{$||e_u||_{H^1}(\%)$}
& \multirow{2}{*}{$||e_{\bar{p} }||_{L^2}(\%)$} \\
(\# iter)  &	& & \\  \hline
452  		& 47.943	&  71.823	&  88.414\\  \hline
\multicolumn{4}{|c|}{without adaptivity } \\ \hline
488 (1)    	& 8.039  	&  24.460	&  21.206\\  \hline
524 (2)    	& 2.266  	&  12.286	&  11.107\\  \hline
596 (4)    	& 0.419  	&  2.477	&  1.433\\  \hline
740 (8)   	&  0.050  	&  0.733	&  0.038\\  \hline
\multicolumn{4}{|c|}{with adaptivity, $\eta^2_i = r^2_i$} \\ \hline
492 (3)  	& 2.444	&  13.528	&  11.355	\\  \hline
534 (6)  	& 0.546	&  4.315	&  3.168	\\  \hline
593 (10)  	& 0.087	&  0.870	&  0.282	\\  \hline
718 (20)  	& 0.041	&  0.501	&  0.025	\\  \hline
\multicolumn{4}{|c|}{with adaptivity, $\eta^2_i = r^2_i \lambda_{i+1}^{-1}$} \\ \hline
511 (2)  	& 2.346	&  12.396	&  10.890	\\  \hline
543 (3)  	& 1.302	&  7.944	&  3.784	\\  \hline
605 (5)  	& 0.175	&  1.157	&  0.443	\\  \hline
768 (11)  	& 0.043	&  0.507	&  0.068	\\  \hline
\end{tabular}
$\;\;\;$
%\caption{One offline basis ($\theta = 0.7$) for perforated domain with big holes. }
%\label{err-ex-online-1}
%\end{table}
%
%
%
%\begin{table}[!h]\label{tab:st-ex-2}
%\centering
\begin{tabular}{ |c | c | c | c | }
     \hline
$DOF$& \multirow{2}{*}{$||e_u||_{L^2}(\%)$}
&\multirow{2}{*}{$||e_u||_{H^1}(\%)$}
& \multirow{2}{*}{$||e_{\bar{p} }||_{L^2}(\%)$} \\
(\# iter)  &	 & &  \\  \hline
694 		& 4.117	&  20.430		&  13.635\\  \hline
\multicolumn{4}{|c|}{without adaptivity } \\ \hline
730 (1)  	& 0.260	&  2.293	&  1.443	\\  \hline
766 (2)  	& 0.075	&  0.982	&  0.057	\\  \hline
838 (4)  	& 0.030	&  0.469	&  0.030	\\  \hline
982 (8)  	& 0.008	&  0.169	&  0.005	\\  \hline
\multicolumn{4}{|c|}{with adaptivity, $\eta^2_i = r^2_i$} \\ \hline
735 (3)  	& 0.085	&  1.100	&  0.070	\\  \hline
766 (5)  	& 0.049	&  0.678	&  0.031	\\  \hline
842 (10)  	& 0.016	&  0.254	&  0.023	\\  \hline
981 (20)  	& 0.006	&  0.133	&  0.002	\\  \hline
\multicolumn{4}{|c|}{with adaptivity, $\eta^2_i = r^2_i \lambda_{i+1}^{-1}$} \\ \hline
762 (2)  	& 0.075	&  0.982	&  0.057	\\  \hline
796 (3)  	& 0.055	&  0.776	&  0.056	\\  \hline
864 (5)  	& 0.024	&  0.388	&  0.014	\\  \hline
1000 (9)  	& 0.007	&  0.149	&  0.004	\\  \hline
\end{tabular}
$\;\;\;$
%\caption{Two offline basis ($\theta = 0.7$) for perforated domain with big holes. }
%\label{err-ex-online-2}
%\end{table}
%
%
%
%
%\begin{table}[!h]\label{tab:st-ex-4}
%\centering
 \begin{tabular}{ |c | c | c | c | }
    \hline
$DOF$& \multirow{2}{*}{$||e_u||_{L^2}(\%)$}
&\multirow{2}{*}{$||e_u||_{H^1}(\%)$}
& \multirow{2}{*}{$||e_{\bar{p} }||_{L^2}(\%)$} \\
(\# iter)  &	 & &  \\  \hline
936	 	& 0.407	&  5.627	&  2.091	\\  \hline
\multicolumn{4}{|c|}{without adaptivity } \\ \hline
972 (1)  	& 0.030	&  0.720	&  0.058	\\  \hline
1008 (2)  	& 0.019	&  0.490	&  0.014	\\  \hline
1080 (4)  	& 0.007	&  0.197	&  0.004	\\  \hline
1224 (8)  	& 0.004	&  0.119	&  0.002	\\  \hline
\multicolumn{4}{|c|}{with adaptivity, $\eta^2_i = r^2_i$} \\ \hline
977 (3)  		& 0.022	&  0.564	&  0.027	\\  \hline
1023 (6)  		& 0.010	&  0.275	&  0.011	\\  \hline
1085 (10)  	& 0.006	&  0.167	&  0.003	\\  \hline
1226 (19)  	& 0.003	&  0.089	&  0.001	\\  \hline
\multicolumn{4}{|c|}{with adaptivity, $\eta^2_i = r^2_i \lambda_{i+1}^{-1}$} \\ \hline
972 (1)  	& 0.030	&  0.720	&  0.058	\\  \hline
1040 (3)  	& 0.012	&  0.303	&  0.009	\\  \hline
1108 (5)  	& 0.006	&  0.161	&  0.003	\\  \hline
1244 (9)  	& 0.003	&  0.092	&  0.002	\\  \hline
\end{tabular}
\caption{Stokes problem for perforated domain with big inclusions(Figure \ref{fig:twodomain}, right). One (Upper left), Two (Upper right) and Three (Bottom) offline basis functions ($\theta = 0.7$). }
\label{tab:st-ex}
\end{table}

%\clearpage
%\newpage

\section{Convergence Analysis}
\label{Analysis}
The result in \cite{chung2015residual} has shown the convergence for online adaptive method applied to elliptic problems, and the same results can be applied for elasticity problem. In this section, we will prove the convergence of adaptive online GMsFEM for Stokes problem.

First, we will prove the following inf-sup condition for the approximation of Stokes problem using offline GMsFEM.
This ensures that the method, with both offline and online basis functions,
is well-posed.
We will assume the continuous inf-sup condition holds.
In particular, there is a constant $C_{\text{cont}}>0$ such that
for any $p\in L^2(\Omega^{\eps})$ with zero mean, we have
\begin{equation}
\label{eq:cont_infsup}
\sup_{v\in (H^1_0(\Omega^{\eps}))^2} \frac{\int_{\Omega^{\epsilon}} \div(v)p}{\|v\|_{H^1{(\Omega^{\epsilon})}}} \geq C_{\text{\rm cont}} \norm{p}_{L^2(\Omega^{\epsilon})}.
\end{equation}
Equivalently, there exists $v\in (H^1_0(\Omega^{\eps}))^2$ such that
$\div v = p$ and $\| v\|_{H^1(\Omega^{\eps})} \leq C_{\text{cont}}^{-1} \|p\|_{L^2(\Omega^{\eps})}$.
%Let $V_{\text{off}}^0$ be the restriction of $V_{\text{off}}$ spanned by all basis function $\Phi_k^{\text{off}}$ corresponding to the interior edges only. Then all functions in $V_{\text{off}}^0$ have zero normal component on the global boundary of the domain.
Let $N^0_e$ be the number of interior coarse edges.
We remark that, for each interior coarse edge $E_i$, there exists a basis function $\Phi_i^{\text{off}}$ such that $\int_{E_i} \Phi_i^{\text{off}}\cdot n_i \neq 0$.

\begin{lemma}
For all $p\in Q_{\text{\rm off}}$, there is a constant $C_{\text{\rm infsup}} > 0$ such that
\begin{equation}
 \sup_{u\in V_{\text{\rm off}}} \frac{\int_{\Omega^{\epsilon}} \div(u)p}{\|u\|_{H^1{(\Omega^{\epsilon})}}} \geq C_{\text{\rm infsup}} \norm{p}_{L^2(\Omega^{\epsilon})}.
 \label{eq:infsup}
\end{equation}
\end{lemma}

\begin{proof}
Let $p\in Q_{\text{off}}$ with zero mean.
Using the continuous inf-sup condition (\ref{eq:cont_infsup}),
there exists $v\in H^1_0(\Omega^{\eps})$ such that
$\div v = p$ and $\| v\|_{H^1(\Omega^{\eps})} \leq C_{\text{cont}}^{-1} \|p\|_{L^2(\Omega^{\eps})}$.
Since, for each interior coarse edge $E_i$, there exists a basis function $\Phi_i^{\text{off}}$ such that $\int_{E_i} \Phi_i^{\text{off}}\cdot n_i \neq 0$.
We can then define $u\in V_{\text{off}}$ by the following
%There exist $\Phi_i^{\text{off}}$ corresponding to the basis on the edges, such that we can define $u\in V_{off}^0$ as follows
\begin{equation*}
 u=\sum_{i=1}^{N_e^0} c_i \Phi_i^{\text{off}}, \quad c_i=\int_{E_i} v \cdot n_i
\end{equation*}
where we assume that the basis function are normalized so that $\int_{E_i} \Phi_{i}^{\text{off}}\cdot n_i = 1$. So, by the Green's identity, we have
\begin{equation*}
\int_{ \Omega^{\epsilon}} p^2 = \int_{ \Omega^{\epsilon}} \div{(v)}  p = \sum_{i=1}^{N^0_e}\int_{E_i} (v \cdot n_i)[p] = \sum_{i=1}^{N^0_e}\int_{E_i} c_i(\Phi_i^{\text{off}}\cdot n_i)[p]= \int_{ \Omega^{\epsilon}} \div {(u)} p
\end{equation*}
where $[p]$ is the jump of $p$.
%This implies $\div u = p$ in $\Omega^{\epsilon}$.
We will next show that there is a constant $C_{\text{\rm infsup}} > 0$
such that $\|u\|_{H^1(\Omega^{\eps})} \leq C_{\text{infsup}}^{-1} \| p\|_{L^2(\Omega^{\eps})}$.

%Next, we want to show $\norm{\nabla u}_{L^2(\Omega^{\epsilon})} \leq \norm{p}_{L^2(\Omega^{\epsilon})}$.
Since $c_i^2 \leq H\int_{E_i} (v \cdot n_i)^2$, we have
\begin{equation*}
\norm{\nabla u}_{L^2(\Omega^{\epsilon})}^2  \leq \sum_{i=1}^{N_e^0}\int_{\omega_i^{\epsilon}} c_i^2 \, \nabla \Phi_{i}^{\text{off}} : \nabla \Phi_{i}^{\text{off}} \leq C_{\text{max}} H \sum_{K\in \mathcal{T}^H} \int_{\partial K} (v\cdot n)^2
\end{equation*}
where $C_{\text{max}} = \max_{1\leq i\leq N_e^0} C_i$ and $C_i = \min \int_{\omega_i^{\epsilon}} \nabla \Phi_i^{\text{off}}  : \nabla \Phi_i^{\text{off}}$
with the minimum taken over all basis functions $\Phi_i^{\text{off}}$ such that $\int_{E_i} \Phi_i^{\text{off}}\cdot n_i \neq 0$. Notice that the constant $C_i$ are independent of the mesh size.
Using the trace theorem on the coarse element $K$, we have $H\int_{\partial K} (v\cdot n)^2\preceq \| v\|_{H^1(K)}^2$.
So, by the continuous inf-sup condition, we obtain
$$ \sum_{K\in\mathcal{T}^H} H\int_{\partial K} (v\cdot n)^2\preceq \|v\|_{H^1(\Omega^{\eps})} \preceq \| p\|^2_{L^2(\Omega^{\eps})}.$$
This completes the proof.
%by taking $C_{\text{infsup}} = C_{\text{max}}^{-\frac{1}{2}}$.

%$C=\max_{1\leq i\leq N_e} \min_{r_i} \int_{\omega_i^{\epsilon}} \nabla \Phi_i^{\text{off}}  \nabla \Phi_i^{\text{off}}$, $r_i$ is the number of basis corresponding to coarse edge $E_i$.

%We can show that  $H\int_{\partial K} (v\cdot n)^2\preceq \norm{\nabla{\eta}}^2_{L^2{(K)}} + \norm{p}^2_{L^2{(K)}}$.

%Thus, we have  $\norm{\nabla u}_{L^2(\Omega^{\epsilon})}^2 \preceq \norm{p}^2_{L^2(\Omega^{\epsilon})}$, here we use the fact that $\norm{\nabla{\eta}}^2_{L^2{(K)}} \preceq \norm{p}^2_{L^2{(K)}}$.
%\begin{equation*}
%C_{\text{infsup}} \int_{ \Omega^{\epsilon}} \div(u) p = C_{\text{infsup}}  \int_{ \Omega^{\epsilon}} p^2 \geq \norm{p}_{L^2(\Omega^{\epsilon})} \norm{\nabla u}_{L^2(\Omega^{\epsilon})}.
%\end{equation*}
%where $C_{\text{infsup}} = (C+1)^{\frac{1}{2}}$.
%Therefore,
%\begin{equation*}
%C_{\text{infsup}} \frac{\int_{\Omega^{\epsilon}} \div(u)p}{\norm{u}_{H^1{(\Omega^{\epsilon})}}} \geq \norm{p}_{L^2(\Omega^{\epsilon})}.
%\end{equation*}
%Take the supreme of the above inequality over $V_{\text{off}}^0$ both sides, we get the desired inf-sup condition.

\end{proof}

Now, we will show the convergence of our online adaptive enrichment scheme for the Stokes problem.
First, we define a reference solution by $(u,p)\in (H^1_0(\Omega^{\eps}))^2 \times Q_{\text{off}}$ which solves
\begin{align}\label{eq:fine_system2}
\avrg{\mathcal {L}^{\eps}(u,p),(v,q)}{\om^{\eps}}= ((f,0),(v,q))_{\om^{\eps}}, \qquad  \text{for all } (v,q) \in H^1_0(\Omega^{\eps})^2 \times Q_\text{off}.
\end{align}
Notice that the solution of (\ref{eq:fine_system2}) and the solution of (\ref{eq:fine_system1}) have a difference proportional to the coarse mesh size $H$.
We also define a snapshot solution by $(\hat{u},\hat{p})\in V_{\text{\rm snap}} \times Q_{\text{off}}$ which solves
\begin{align}\label{eq:fine_system3}
\avrg{\mathcal {L}^{\eps}(\hat{u},\hat{p}),(v,q)}{\om^{\eps}}= ((f,0),(v,q))_{\om^{\eps}}, \qquad  \text{for all } (v,q) \in V_{\text{\rm snap}} \times Q_\text{off}.
\end{align}
We notice that the difference $\|u-\hat{u}\|_{H^1(\Omega^{\eps})}$ represents an irreducible error.
%and we use (\ref{eq:fine_system2}) to simplifies the analysis.
%We remark that we use (\ref{eq:fine_system2}) to simplifies the analysis.
Furthermore, standard finite element analysis shows that
\begin{equation}
\label{eq:cea}
\| u - u_{\text{\rm ms}}\|_{H^1(\Omega^{\eps})} \leq \| u - \widetilde{u}_{\text{\rm ms}}\|_{H^1(\Omega^{\eps})}
\end{equation}
for any $\widetilde{u}_{\text{ms}} \in V_{\text{off}}$.
%Recall that $(u, p) \in V_h \times Q_h$ is the fine-grid solution of the Stokes equations.
Next, we prove the following a-posteriori error bound for the offline GMsFEM (\ref{eq:coarse_system}).
The notation $a \preceq b$ means that there is a generic constant $C>0$ such that $a\leq Cb$.

\begin{theorem}
Let $u$ be the reference solution defined in (\ref{eq:fine_system2}), $\hat{u}$ be the snapshot solution defined in (\ref{eq:fine_system3})
and $u_{\text{\rm ms}}$ be the multiscale solution satisfying (\ref{eq:coarse_system}). Then,
%there exists $\hat{u}_{\text{\rm ms}} \in V_{\text{\rm off}}$ such that
we have
\begin{align}
\label{eq:post1}
\norm{ \hat{u} - u_{\text{\rm ms}}}_{H^1(\Omega^{\epsilon})}^2
\leq C_{\text{s}}  \sum_{i=1}^{N_u}  \big( 1 + \frac{1}{\lambda^{i,\text{\rm off}}_{l_i+1}} \big) \norm{R_i}_{V^*}^2
\end{align}
where $l_i$ is the number of offline basis functions used for the coarse neighborhood $\omega_i^{\eps}$,
and $\lambda^{i,\text{\rm off}}_j$ is the $j$-th eigenvalue for the coarse neighborhood $\omega_i^{\eps}$.
The constant $C_s$ is the maximum number of coarse neighborhoods corresponding to coarse blocks.
Moreover, we have
\begin{align}
\label{eq:post2}
\norm{ u - u_{\text{\rm ms}}}_{H^1(\Omega^{\epsilon})}^2
\leq 2C_s  \sum_{i=1}^{N_u}  \big( 1 + \frac{1}{\lambda^{i,\text{\rm off}}_{l_i+1}} \big) \norm{R_i}_{V^*}^2  + 2\norm{ u - \hat{u}}_{H^1(\Omega^{\epsilon})}^2 .
\end{align}
 \end{theorem}

\begin{proof}
For any $\phi \in V_{\text{snap}}$ such that $\int_{K_i^{\epsilon}} \div \phi =  0$ and $\phi = \mathcal{H}(\phi)$, we have
\begin{align}
\label{before}
\int_{\Omega^{\epsilon}} \nabla(\hat{u} - u_{\text{ms}}) :\nabla \phi
&= \int_{\Omega^{\epsilon}} \nabla(\hat{u} - u_{\text{ms}}) : \nabla \phi  - \int_{\Omega^{\epsilon}} (\hat{p} - p_{\text{ms}}) \div \phi
%&=  \int_{\Omega^{\epsilon}} \nabla(\hat{u} - u_{\text{ms}}) : \nabla \phi  - \int_{\Omega^{\epsilon}} (p - p_{\text{ms}}) \div \phi +\int_{\Omega^{\epsilon}} (p - \overline{p} ) \div \phi, \numberthis\label{before}
\end{align}
where we use the fact that $\int_{K_i^{\epsilon}} (\hat{p} - p_{\text{ms}}) \div \phi = 0$ since $\hat{p} - p_{\text{ms}}$ is constant in $K_i^{\epsilon}$.
We can write (\ref{before}) as
\begin{equation}\label{after}
\int_{\Omega^{\epsilon}} \nabla(\hat{u} - u_{\text{ms}}) : \nabla \phi  = R(\phi)
%+  \int_{\Omega^{\epsilon}} \nabla(\hat{u} - u) : \nabla \phi
%+ \int_{\Omega^{\epsilon}} (p - \overline{p} ) \div \phi
\end{equation}
where $R(\phi)$ is the global residual defined by
$R(\phi) = \int_{\Omega^{\epsilon}} \nabla(u - u_{\text{ms}}) : \nabla \phi  - \int_{\Omega^{\epsilon}} (p - p_{\text{ms}}) \div \phi$ for all $\phi$.
Let $\phi^{\text{off}}$ be an arbitrary function in the space $V_{\text{off}}$.
We can write $\phi^{\text{off}} = \sum_{i=1}^{N_u} \phi^{\text{off}}_i$ where $\phi^{\text{off}}_i$
is the component of $\phi^{\text{off}}$ in the local offline space corresponding to the coarse neighborhood $\omega_i^{\eps}$.
Using the facts that $V_{\text{off}} \subset V_{\text{snap}}$ and $R(\phi^{\text{off}})=0$, we can write $R(\phi)$ as
\begin{align*}
R(\phi) = R(\mathcal{H}(\phi) - \phi^{\text{off}}) =  R\Big(\sum_{i=1}^{N_u} (\mathcal{H}(\chi_i\phi) - \phi^{\text{off}}_i) \Big)
=  \sum_{i=1}^{N_u} R_i\Big(\mathcal{H}(\chi_i\phi) - \phi^{\text{off}}_i \Big) \numberthis  \label{trunc}
\end{align*}
 where $R_i$ is the local residual defined in \eqref{stokes-resid}.
We will define $\phi^{\text{off}}$ as follows.
Notice that $\mathcal{H}(\chi_i\phi)$ belongs to the local snapshot space $V_{\text{snap}}^i$.
We can take $\phi^{\text{off}}_i$ as the component of $\mathcal{H}(\chi_i\phi)$
in the offline space $V_{\text{off}}^i$. We write $\phi^{\text{off}}_i = \mathcal{H}(\chi_i \phi_i)$.

Then from \eqref{trunc}, we have
\begin{equation*}
R(\phi) \leq \sum_{i=1}^{N_u} \norm{R_i}_{(V^i)^*} \, \norm{ \mathcal{H}(\chi_i \phi) - \mathcal{H}(\chi_i \phi_i) }_{H^1(\omega_i^{\eps})}.
\end{equation*}
Using the minimum energy property, we have
\begin{equation*}
R(\phi) \leq \sum_{i=1}^{N_u} \norm{R_i}_{(V^i)^*} \, \norm{ \chi_i (\phi - \phi_i) }_{H^1(\omega_i^{\eps})}.
\end{equation*}
By the spectral problem (\ref{eq:offeig2}), we obtain
\begin{equation}
\label{eq:residualbound}
R(\phi) \leq \sum_{i=1}^{N_u} \big( 1 + \frac{1}{\lambda^{i,\text{off}}_{l_i+1}} \big)^{\frac{1}{2}} \norm{R_i}_{(V^i)^*} \, \norm{ \phi - \phi_i }_{H^1(\omega_i^{\eps})}
\leq \sum_{i=1}^{N_u} \big( 1 + \frac{1}{\lambda^{i,\text{off}}_{l_i+1}} \big)^{\frac{1}{2}} \norm{R_i}_{(V^i)^*} \, \norm{ \phi }_{H^1(\omega_i^{\eps})}
\end{equation}
where we used the orthogonality of eigenfunctions from the spectral problem (\ref{eq:offeig2}).
Finally, we take $\phi = \hat{u} - u_{\text{ms}}$. Notice that, by (\ref{eq:fine_system3}) and (\ref{eq:coarse_system}),
we have
\begin{equation*}
\int_{K_i^{\eps}} \div (\hat{u}-u_{\text{ms}}) = 0.
\end{equation*}
In addition, for this choice of $\phi$, we have $\phi = \mathcal{H}(\phi)$ since $\hat{u}, u_{\text{ms}} \in V_{\text{snap}}$.
Hence (\ref{after}) and (\ref{eq:residualbound}) imply that
\begin{equation*}
\| \hat{u} - u_{\text{ms}} \|_{H^1(\Omega^{\eps})}^2  \leq \sum_{i=1}^{N_u} \big( 1 + \frac{1}{\lambda^{i,\text{off}}_{l_i+1}} \big)^{\frac{1}{2}} \norm{R_i}_{(V^i)^*} \, \norm{ \hat{u}-u_{\text{ms}} }_{H^1(\omega_i^{\eps})},
%+  \int_{\Omega^{\epsilon}} \nabla(\hat{u} - u) : \nabla \phi
%+ \int_{\Omega^{\epsilon}} (p - \overline{p} ) \div \phi
\end{equation*}
which shows (\ref{eq:post1}).
The proof for (\ref{eq:post2}) follows from $\| u - u_{\text{ms}}\|_{H^1(\Omega^{\eps})}
\leq \| u - \hat{u}\|_{H^1(\Omega^{\eps})}  + \| \hat{u} - u_{\text{ms}}\|_{H^1(\Omega^{\eps})}$.

\end{proof}

We recall that the norm of the local residual $R_i$ is defined in (\ref{eq:Rn1}).
We define a modified norm as
\begin{equation}
\label{eq:Rn2}
||R_i||_{(V_0^i)^*} = \sup_{v \in V_0^i } \frac{|R_i(v)|}{\|v\|_{H^1(\omega_i^{\eps})}}
\end{equation}
where $V_0^i \subset V^i$ and the vectors $v\in V_0^i$ satisfies $\int_{\Omega^{\eps}} \div(v) \, q = 0$ for all $q\in Q_{\text{off}}$.
It is easy to show that $||R_i||_{(V_0^i)^*} \leq ||R_i||_{(V^i)^*}$.
In the next theorem, we will show the convergence
of the online adaptive GMsFEM for the Stokes problem.
The theorem states that our method is convergent up to an irreducible error $\|u-\hat{u}\|_{H^1(\Omega^{\eps})}$
with enough number of offline basis functions.

\begin{theorem}
\label{thm:online}
Let $u$ be the reference solution defined in (\ref{eq:fine_system2}), $\hat{u}$ be the snapshot solution defined in (\ref{eq:fine_system3})
and $u_{\text{\rm ms}}^m$ be the multiscale solution
of (\ref{eq:coarse_system}) in the enrichment level $m$.
Assume that $l_i$ offline basis functions for the coarse neighborhood $\omega_i^{\eps}$
are used as initial basis in the online procedure.
Suppose that one online basis is added to a single coarse neighborhood $\omega_i^{\eps}$.
Then, there is a constant $D$ such that
\begin{equation}
\label{eq:online_conv}
\| u - u_{\text{\rm ms}}^{m+1}\|_{H^1(\Omega^{\eps})}^2
\leq (1+\delta_3)(1+\delta_2) \Big(1+\delta_1 - \theta C_s^{-1} \frac{\lambda^{i,\text{\rm off}}_{l_i+1}}{\lambda^{i,\text{\rm off}}_{l_i+1}+1}\Big) \| \hat{u} - u_{\text{\rm ms}}^m\|_{H^1(\Omega^{\eps})}^2
+ D \| u - \hat{u}\|_{H^1(\Omega^{\eps})}^2
\end{equation}
where $\delta_1,\delta_2,\delta_3 >0$ are arbitrary and $D$ depends only on $\delta_i, i=1,2,3$. In addition, $\theta$ is the relative residual defined by
\begin{equation*}
\theta = ||R_i||_{(V_0^i)^*}^2 \Big/ \sum_{i=1}^{N_u} \norm{R_i}_{(V^i)^*}^2.
\end{equation*}
\end{theorem}

\begin{proof}
We will first consider the addition of only one online basis function $\phi_i^{\text{on}}$ to the space $V_{\text{off}}^m$.
For any function $\widetilde{u}_{\text{ms}} \in V_{\text{off}}^{m+1}$, by (\ref{eq:cea}), we have
\begin{equation}
\label{eq:proof0}
\| u - u_{\text{\rm ms}}^{m+1}\|_{H^1(\Omega^{\eps})} \leq \| u - \widetilde{u}_{\text{\rm ms}}\|_{H^1(\Omega^{\eps})}
\leq \| \hat{u} - \widetilde{u}_{\text{\rm ms}}\|_{H^1(\Omega^{\eps})} + \| u -\hat{u} \|_{H^1(\Omega^{\eps})}.
\end{equation}
We will derive an estimate for $\| \hat{u} - \widetilde{u}_{\text{\rm ms}}\|_{H^1(\Omega^{\eps})}$.
We take $\widetilde{u}_{\text{ms}} = u_{\text{ms}}^m + \alpha \phi_i^{\text{on}}$ where $\alpha$ is a scalar to be determined.
Then we have
\begin{equation*}
\| \hat{u} - \widetilde{u}_{\text{\rm ms}}\|_{H^1(\Omega^{\eps})}^2
= \| \hat{u} - u_{\text{\rm ms}}^m\|_{H^1(\Omega^{\eps})}^2 - 2\alpha \int_{\omega^{\eps}_i} \nabla (\hat{u}-u_{\text{ms}}^m) : \nabla \phi_i^{\text{on}}
+ \alpha^2 \| \phi_i^{\text{on}}\|_{H^1(\Omega^{\eps})}^2.
\end{equation*}
Using the definition of the residual $R_i$ and the fact that $\int_{\omega_i^{\eps}} \div (\phi_i^{\text{on}}) \, q = 0$ for all $q\in Q_{\text{off}}$, we have
\begin{equation*}
\| \hat{u} - \widetilde{u}_{\text{\rm ms}}\|_{H^1(\Omega^{\eps})}^2
= \| \hat{u} - u_{\text{\rm ms}}^m\|_{H^1(\Omega^{\eps})}^2 - 2\alpha R_i( \phi_i^{\text{on}})
+ \alpha^2 \| \phi_i^{\text{on}}\|_{H^1(\Omega^{\eps})}^2
+ 2\alpha \int_{\omega^{\eps}_i} \nabla (u-\hat{u}) : \nabla \phi_i^{\text{on}}.
\end{equation*}
%We normalize the online basis function $\phi_i^{\text{on}}$ so that $\| \phi_i^{\text{on}}\|_{H^1(\Omega^{\eps})}=1$.
Taking $\alpha = R_i(\phi_i^{\text{on}})/\| \phi_i^{\text{on}}\|_{H^1(\omega_i^{\eps})}^2 $, we have
\begin{equation}
\label{eq:proof1}
\| \hat{u} - \widetilde{u}_{\text{\rm ms}}\|_{H^1(\Omega^{\eps})}^2
= \| \hat{u} - u_{\text{\rm ms}}^m\|_{H^1(\Omega^{\eps})}^2 -  \frac{R_i( \phi_i^{\text{on}})^2}{\| \phi_i^{\text{on}}\|_{H^1(\omega_i^{\eps})}^2}
+ 2\alpha \int_{\omega^{\eps}_i} \nabla (u-\hat{u}) : \nabla \phi_i^{\text{on}}.
\end{equation}

Using (\ref{online_Stokes}), we have
\begin{equation}
\label{eq:proof2}
R_i(v) = \int_{\omega_i^{\eps}} \nabla \phi_i^{\text{on}} : \nabla v, \quad \forall v\in V^i_0.
\end{equation}
By (\ref{eq:Rn2}) and (\ref{eq:proof2}), we have $||R_i||_{(V_0^i)^*} \leq \| \phi_i^{\text{on}} \|_{H^1(\omega_i^{\eps})}$.
Taking $v=\phi_i^{\text{on}}$ in (\ref{eq:proof2}), we have $R_i(\phi_i^{\text{on}}) = \| \phi_i^{\text{on}}\|_{H^1(\omega_i^{\eps})}^2$.
Thus, (\ref{eq:proof1}) becomes
\begin{equation}
\label{eq:proof3}
\| \hat{u} - \widetilde{u}_{\text{\rm ms}}\|_{H^1(\Omega^{\eps})}^2
= \| \hat{u} - u_{\text{\rm ms}}^m\|_{H^1(\Omega^{\eps})}^2 -  ||R_i||_{(V_0^i)^*}^2
+ 2\alpha \int_{\omega^{\eps}_i} \nabla (u-\hat{u}) : \nabla \phi_i^{\text{on}}.
\end{equation}
Using the definition of $\theta$ and (\ref{eq:post1}), we have
\begin{equation}
\label{eq:proof4}
\| \hat{u} - \widetilde{u}_{\text{\rm ms}}\|_{H^1(\Omega^{\eps})}^2
\leq \Big(1 - \theta C_s^{-1} \frac{\lambda^{i,\text{off}}_{l_i+1}}{\lambda^{i,\text{off}}_{l_i+1}+1}\Big) \| \hat{u} - u_{\text{\rm ms}}^m\|_{H^1(\Omega^{\eps})}^2
+ 2\alpha \int_{\omega^{\eps}_i} \nabla (u-\hat{u}) : \nabla \phi_i^{\text{on}}.
\end{equation}
The last term in (\ref{eq:proof4}) can be estimated as
\begin{equation*}
2\alpha \int_{\omega^{\eps}_i} \nabla (u-\hat{u}) : \nabla \phi_i^{\text{on}}
\leq 2 \| u - \hat{u}\|_{H^1(\Omega^{\eps})} \, \frac{R_i(\phi_i^{\text{on}})}{\|\phi_i^{\text{on}}\|_{H^1(\omega_i^{\eps})}}
\end{equation*}
Using the definition of $R_i$, we have $R_i(\phi_i^{\text{on}}) = \int_{\omega_i^{\eps}} \nabla ( u - u_{\text{ms}}^m ) : \nabla\phi_i^{\text{on}}$. So,
\begin{equation*}
2\alpha \int_{\omega^{\eps}_i} \nabla (u-\hat{u}) : \nabla \phi_i^{\text{on}}
\leq 2 \| u - \hat{u}\|_{H^1(\Omega^{\eps})} \, \| u - u_{\text{ms}}^m \|_{H^1(\Omega^{\eps})}.
\end{equation*}
Notice that $2 \| u - \hat{u}\|_{H^1(\Omega^{\eps})} \, \| \hat{u} - u_{\text{ms}}^m \|_{H^1(\Omega^{\eps})}
\leq \delta_1^{-1} \| u - \hat{u}\|_{H^1(\Omega^{\eps})}^2 + \delta_1 \| \hat{u} - u_{\text{ms}}^m \|_{H^1(\Omega^{\eps})}^2$
for any $\delta_1>0$.
Therefore, (\ref{eq:proof4}) becomes
\begin{equation}
\label{eq:proof5}
\| \hat{u} - \widetilde{u}_{\text{\rm ms}}\|_{H^1(\Omega^{\eps})}^2
\leq \Big(1+\delta_1 - \theta C_s^{-1} \frac{\lambda^{i,\text{off}}_{l_i+1}}{\lambda^{i,\text{off}}_{l_i+1}+1}\Big) \| \hat{u} - u_{\text{\rm ms}}^m\|_{H^1(\Omega^{\eps})}^2
+ (2+\delta_1^{-1}) \| u - \hat{u}\|_{H^1(\Omega^{\eps})}^2.
\end{equation}

Finally, combining (\ref{eq:proof0}) and (\ref{eq:proof5}), we have
\begin{equation}
\| u - u_{\text{\rm ms}}^{m+1}\|_{H^1(\Omega^{\eps})}^2
\leq (1+\delta_2) \Big(1+\delta_1 - \theta C_s^{-1} \frac{\lambda^{i,\text{off}}_{l_i+1}}{\lambda^{i,\text{off}}_{l_i+1}+1}\Big) \| \hat{u} - u_{\text{\rm ms}}^m\|_{H^1(\Omega^{\eps})}^2
+ (3+\delta_1^{-1}+\delta_2^{-1}) \| u - \hat{u}\|_{H^1(\Omega^{\eps})}^2.
\end{equation}
We obtain the desired result by noting
that $$\| \hat{u} - u_{\text{\rm ms}}^m\|_{H^1(\Omega^{\eps})}^2 \leq (1+\delta_3) \| u - u_{\text{ms}}^m \|_{H^1(\Omega^{\eps})}^2
+ (1+\delta_3^{-1}) \| u - \hat{u} \|_{H^1(\Omega^{\eps})}^2$$
for any $\delta_3>0$.

\end{proof}

We remark that, in order to obtain rapid convergence, one needs to choose $l_i$
large enough so that $\lambda^{i,\text{\rm off}}_{l_i+1}$ is large. In this case, the quantity
$\lambda^{i,\text{\rm off}}_{l_i+1}/(\lambda^{i,\text{\rm off}}_{l_i+1}+1)$ is close to one.
Then, (\ref{eq:online_conv}) shows that the resulting online adaptive enrichment procedure has a rapid convergence.

Theorem \ref{thm:online} gives the convergence of our online adaptive enrichment procedure
when one online basis is added at a time.
One can also add online basis in non-overlapping coarse neighborhoods.
Using the same proof as Theorem \ref{thm:online}, we obtain the following result.

\begin{theorem}
Let $u$ be the reference solution defined in (\ref{eq:fine_system2}), $\hat{u}$ be the snapshot solution defined in (\ref{eq:fine_system3})
and $u_{\text{\rm ms}}^m$ be the multiscale solution
of (\ref{eq:coarse_system}) in the enrichment level $m$.
Assume that $l_i$ offline basis functions for the coarse neighborhood $\omega_i^{\eps}$
are used as initial basis in the online procedure.
Let $S$ be the index set for the non-overlapping coarse neighborhoods where online basis functions are added.
Then, there is a constant $D$ such that
\begin{equation}
\label{eq:online_conv1}
\| u - u_{\text{\rm ms}}^{m+1}\|_{H^1(\Omega^{\eps})}^2
\leq (1+\delta_3)(1+\delta_2) \Big(1+\delta_1 - \theta C_s^{-1} \min_{j\in S} \frac{\lambda^{i,\text{\rm off}}_{l_j+1}}{\lambda^{i,\text{\rm off}}_{l_j+1}+1}\Big) \| \hat{u} - u_{\text{\rm ms}}^m\|_{H^1(\Omega^{\eps})}^2
+ D \| u - \hat{u}\|_{H^1(\Omega^{\eps})}^2
\end{equation}
where $\delta_1,\delta_2,\delta_3 >0$ are arbitrary and $D$ depends only on $\delta_i, i=1,2,3$. In addition, $\theta$ is the relative residual defined by
\begin{equation*}
\theta = \sum_{i\in S} ||R_i||_{(V_0^i)^*}^2 \Big/ \sum_{i=1}^{N_u} \norm{R_i}_{(V^i)^*}^2.
\end{equation*}

\end{theorem}

The above result suggests that adding more online basis functions at each iteration
will speed up the convergence.
Lastly, we remark that the convergence for the pressure can be obtained using the inf-sup condition (\ref{eq:infsup}).

\section{Conclusion}
\label{conclusion}

We present an efficient multiscale procedure for solving PDEs in perforated domains. We consider elliptic, elastic, and Stokes systems.
In our previous work \cite{CELV2015}, we presented a first step in constructing
the offline multiscale basis functions
(without analysis) for solving PDEs in perforated domains.
It is known that
the convergence of multiscale methods can be significantly accelerated
if appropriate online basis functions are constructed and appropriate number
of offline basis functions are used. The construction of online basis functions
relies on analysis and the choice of the offline basis functions.
In this paper, we (1) develop analysis for GMsFEM for perforated domains
(2) design procedures for constructing online multiscale basis functions
(3) present analysis of online multiscale procedures (4) develop adaptive procedures (5) present numerical
results.
By using a computable error indicator, we locate regions,
 where enrichment is necessary,
and construct new online basis functions in order to improve the accuracy.
Our numerical results for the elasticity equation and the Stokes system
show that the method has an excellent performance and rapid convergence.
In particular, only a few online basis functions in some selected regions
improve the accuracy of the solution.
%Furthermore, we have presented a convergence analysis of our online adaptive method for the Stokes system.
Our analysis shows that the convergence rate depends on the number of offline basis functions, and one can
obtain a fast convergence by including enough offline basis functions.
This convergence theory can also be applied to the Laplace equation and the elasticity equation.
One possible future direction is the goal-oriented adaptivity \cite{chung2015goal},
in which basis functions are added in order to reduce the goal error.

\section{Acknowledgement}
YE's work is partially supported by the U.S. Department of Energy Office of Science, Office of Advanced Scientific Computing Research, Applied Mathematics program under Award Number DE-FG02-13ER26165 and  the DoD Army ARO Project and NSF (DMS 0934837 and DMS 0811180).
Eric Chung's research is partially supported by Hong Kong RGC General Research Fund (Project: 400813)
and CUHK Faculty of Science Research Incentive Fund 2015-16.
MV's  work is partially supported by Russian Science Foundation Grant RS 15-11-10024 and  RFBR 15-31-20856.

\bibliographystyle{siam}
\bibliography{references}

\end{document}